\documentclass[12pt]{article}
\pdfoutput=1
\usepackage[margin=1in]{geometry}
\usepackage{parskip}
\title{Lattice Minors and Eulerian Posets}
\date{\small University of Minnesota\\ \tt gust1020@umn.edu}
\author{William Gustafson}
\usepackage{amsfonts, amsmath, amssymb, amsthm, tikz}
\usepackage{hyperref}
\usepackage{subcaption}
\usepackage{cleveref}

\DeclareMathOperator{\M}{M}
\newcommand{\mcE}{\mathcal{E}}

\DeclareMathOperator{\irr}{irr}

\DeclareMathOperator{\rk}{rk}
\DeclareMathOperator{\pyr}{Pyr}
\DeclareMathOperator{\prism}{Prism}
\DeclareMathOperator{\link}{link}
\DeclareMathOperator{\zip}{zip}
\newcommand{\zerohat}{{\widehat{0}}}
\newcommand{\onehat}{{\widehat{1}}}
\newcommand{\genlatt}{generator-enriched lattice}
\newcommand{\genlatts}{generator-enriched lattices}

\newcommand{\Genlatts}{Generator-enriched lattices}

\newcommand{\wout}{\setminus}
\newcommand{\cv}{\textup{\textbf{c}}}
\newcommand{\dv}{\textup{\textbf{d}}}
\newcommand{\join}{\vee}

\newcommand{\bigjoin}{\bigvee}

\newcommand{\abs}[1]{{\lvert{#1}\rvert}}
\newcommand{\mcA}{\mathcal{A}}
\newtheorem{theorem}{Theorem}[section]

\newtheorem{lemma}[theorem]{Lemma}
\newtheorem{proposition}[theorem]{Proposition}
\newtheorem{corollary}[theorem]{Corollary}
\newtheorem{conjecture}[theorem]{Conjecture}
\theoremstyle{definition}
\newtheorem{definition}[theorem]{Definition}

\newtheorem{example}[theorem]{Example}
\newtheorem{remark}[theorem]{Remark}

\newcommand{\eqlabel}[1]{\textup{(\refstepcounter{equation}\label{#1}\theequation)}}
\numberwithin{equation}{section}
\begin{document}

	\maketitle

\begin{abstract}
	We introduce posets of simple vertex labeled minors of graphs
	and a generalization to the level of polymatroids, collectively termed minor
	posets.
	We show that any minor poset is isomorphic
	to the face poset of a regular CW sphere, and in particular,
	is Eulerian. We establish \cv\dv-index inequalities induced
	by strong maps, a tight upper bound for \cv\dv-indices of minor
	posets and a tight lower bound for \cv\dv-indices of minor
	posets arising from lattices of maximal length. 
\end{abstract}

\section{Introduction}
Consider a finite graph with labeled vertices and unlabeled edges.
There is a natural partial
ordering on the simple minors in which the cover relations correspond
to deletions and contractions. \Cref{3-cycle minor poset fig} shows the
minor poset of the cycle of length three with an artificial
minimum $\zerohat$ adjoined.

\begin{figure}
	\centering
	\includegraphics[height=0.4\textheight]{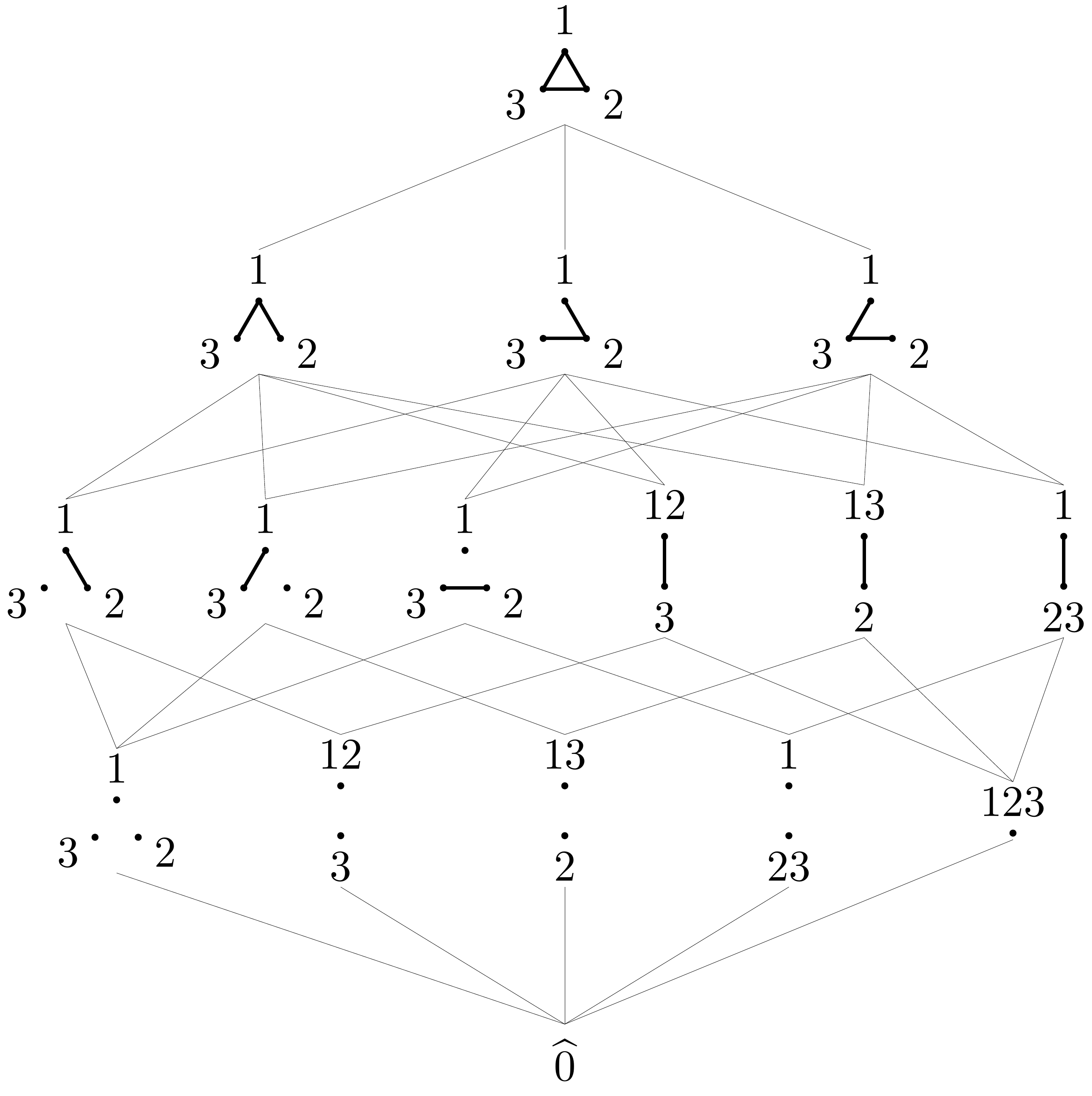}
	\caption{The poset of simple vertex labeled minors of a cycle
	of length three.}
	\label{3-cycle minor poset fig}
\end{figure}

One place where minor posets of graphs show up naturally is as lower
intervals in the uncrossing
poset. This is a partial order on pairings where cover relations correspond
to resolving crossings. This partial order is intimately related with
minors of graphs. Initial interest in the uncrossing poset stemmed from
its relation to a stratification of a space of electrical
networks \cite{electroids}. Each pairing in the poset corresponds to an
equivalence class of graphs and the
relations correspond to either deletions or contractions of edges. In the
case of a circular planar graph -- a graph that can be embedded into the
disc in the plane with all vertices on
the boundary of the disc -- the simple vertex labeled minors may be used
as representatives and the minor poset is isomorphic to a lower interval
in the uncrossing poset. For interested readers the details of
this correspondence appear
in~\cite[Section 1.2.1, Proposition 2.4.5]{thesis}.

Hersh and Kenyon used a lexicographic shelling to show the uncrossing
poset is isomorphic
to the face poset of a regular CW complex in [12, Corollary 3.19]. Thus,
the minor poset of a
circular planar graph is isomorphic to the face poset of a regular CW
sphere.

In this paper we introduce the minor poset of a {\genlatt}. This
is an extension of minor posets of graphs to the level of polymatroids.
Polymatroids are a generalization of matroids for which the rank function
may take on non-integral values \cite{edmonds,herzog-hibi-02}.
{\Genlatts} correspond to closure operators
of simple polymatroids just as geometric lattices correspond to simple matroids;
see \cite[Theorem 2.7]{latticeMinorsAndPolymatroids}. From this relationship
there is a notion of minors for {\genlatts}. In the case of the lattice
of flats of a graph, lattice minors correspond to the simple vertex labeled minors of the graph.

Our main result is a collapsing construction that shows any minor poset is
isomorphic to the face poset of a CW sphere. In particular,
minor posets form a new class of Eulerian posets that are combinatorially
and algebraically motivated. We also study the flag vectors of minor posets
through the lens of the \cv\dv-index; a combinatorial invariant for Eulerian
posets that encodes the flag vectors with all linear redundancies removed.

Although this paper is self-contained, interested readers may wish to
consult \cite{latticeMinorsAndPolymatroids} for the connection between
{\genlatts} and polymatroids. Briefly, a {\genlatt} encodes the closure
operator of a polymatroid in the same way a geometric lattice
encodes the closure operator of a matroid. In the case of graphs, minors
of the lattice of flats are in bijection with simple graph minors when we
consider vertices to be labeled and edges unlabeled.

In \Cref{lattice minors section} we review the notions of {\genlatts} and minors introduced in
\cite{latticeMinorsAndPolymatroids}. In \Cref{strong minor poset section} we introduce the minor
poset of a {\genlatt}. \Cref{factoring section} describes a process to factor strong
surjections used for the collapsing construction. In
\Cref{CW sphere section} we prove our main result, the aforementioned collapsing construction for minor
posets. In \Cref{inequality section} we deduce inequalities for
\cv\dv-indices of minor
posets from the collapsing construction. Appendix A contains an overview
of the zipping operation used in \Cref{CW sphere section} as well as
background concerning the \cv\dv-index.
\section{Lattice minors}
\label{lattice minors section}

Here we review some definitions and results concerning {\genlatts} which were introduced
in \cite{latticeMinorsAndPolymatroids}. In particular, we define the deletion and contraction
operations of {\genlatts} that are analogous to the same operations on simple vertex labeled graphs.

\begin{definition}
	A \emph{\genlatt} is a pair~$(L,G)$ such that~$L$
	is a finite lattice
	and~$G\subseteq L\setminus\{\widehat{0}\}$ generates the lattice~$L$ via the join operation.
\end{definition}

Given a {\genlatt}~$(L,G)$, the elements of~$G$
will be referred to as
\emph{generators} of~$(L,G)$.
Necessarily,~$G$ includes the set of join
irreducibles of~$L$, which we denote by~$\irr(L)$.
When~$G=\irr(L)$
the {\genlatt}~$(L,G)$ is said to be \emph{minimally generated}.
Note a {\genlatt} may have an empty generating set. This
is the trivial case where the lattice only consists of the minimum.
Throughout, we consider the empty join to be equal to the minimum of a lattice, which is the identity of the join operation.

Given a lattice~$L$, let~$H\subseteq L\setminus\{\widehat{0}\}$ and let~$z\in L$ be an element
such that~$z<h$ for all~$h\in H$. Define the \emph{{\genlatt} with generating
set~$H$ and minimum~$z$} to be
\begin{align}
\label{span notation}
\langle H|z\rangle&=\left(\left\{z\vee\bigjoin_{x\in X}x:X\subseteq H\right\},H\right)
\end{align}

\begin{figure}
	\centering
	\hfill
	\begin{subfigure}[t]{0.4\textwidth}
		\centering
		\includegraphics{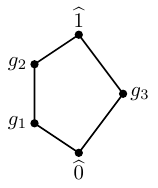}
		\caption{The Hasse diagram of a lattice~$L$.}
	\end{subfigure}
	\hfill
	\begin{subfigure}[t]{0.4\textwidth}
		\centering
		\includegraphics{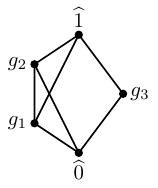}
		\caption{The diagram of a {\genlatt} $(L,\{g_1,g_2,g_3\})$.}
	\end{subfigure}
	\hfill
	\caption{}
	\label{fig cayleys}
\end{figure}

When listing~$H$ explicitly usually the set brackets will be omitted.

In order to depict {\genlatts}, we add extra edges to the Hasse diagram of the given lattice.

\begin{definition}
	The \emph{diagram} of a {\genlatt}~$(L,G)$ is the simple directed graph with vertex set~$L$
	and with edges~$(\ell,\ell\join g)$ for~$\ell\in L$ and~$g\in G$ such that~$g\not\le\ell$.
\end{definition}

Just as in the case of Hasse diagrams edges in diagrams of {\genlatts}
will not be explicitly directed. It is understood all edges
are directed upwards.
Figure~\ref{fig cayleys} depicts an example of the diagram of a {\genlatt} and \Cref{10_genlatts} depicts
the ten {\genlatts} with~3 generators considered up to isomorphism.
The diagram of a {\genlatt} determines the {\genlatt} up to isomorphism. The lattice
is isomorphic to the poset of vertices of the diagram with the order relation~$\ell\le k$ when
there is a directed path from~$\ell$ to~$k$. In particular, the minimum of the lattice
is the unique source vertex. The generating set consists of all vertices incident to the
minimum.

Strong maps between matroids were introduced independently
by Crapo in \cite[Section 2]{crapo-strongmaps} and Higgs
in \cite[pp. 1]{higgs}. We generalize this concept below to
{\genlatts}, the classical case is that of strong maps between
minimally generated geometric lattices.

\begin{definition}
	Let~$(L,G)$ and~$(K,H)$ be two {\genlatts}. A \emph{strong map}
	is a map~$f:L\rightarrow K$
	that is join-preserving
	and satisfies~$f(G)\subseteq H\cup\{\widehat{0}_K\}$.
	We use the notation~$f:(L,G)\rightarrow(K,H)$ to denote a strong map~$f$
	from~$(L,G)$ to~$(K,H)$.
\end{definition}

A strong map~$f:(L,G)\rightarrow(K,H)$ is said to be \emph{injective}
when it is injective as a map on the underlying lattices,
and \emph{surjective} when~$f(G\cup\{\widehat{0}_L\})=H\cup\{\widehat{0}_K\}$.
Two {\genlatts} are said to be \emph{isomorphic} when there is a strong
bijection between them.

Given a finite set~$X$ let~$B_X$ denote the lattice of subsets of~$X$. We will also use the
notation~$B_n$ to denote the lattice of subsets of~$[n]=\{1,\dots,n\}$.
\begin{definition}
\label{canonical strong map def}
	Given a {\genlatt}~$(L,G)$ the \emph{canonical strong map}
	from~$(B_G,\irr(B_G))$ onto~$(L,G)$ is the strong map~$\theta:(B_G,\irr(B_G))
	\rightarrow(L,G)$ defined by
	\[\theta(H)=\bigjoin_{h\in H}h.\]
\end{definition}

\newcommand{\glsize}{0.15}
\newcommand{\glfig}[1]{%
	\begin{subfigure}{\glsize\textwidth}%
		\centering%
		\includegraphics[scale=0.7]{#1.pdf}%
		\caption{}%
		\label{#1}%
	\end{subfigure}%
	}
\begin{figure}
\centering
	\glfig{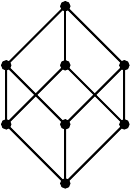}
	\hspace{0.5em}
	\glfig{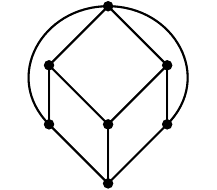}
	\hspace{0.5em}
	\glfig{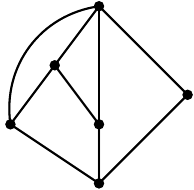}
	\hspace{0.5em}
	\glfig{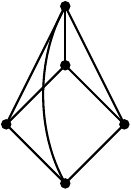}
	\hspace{0.5em}
	\glfig{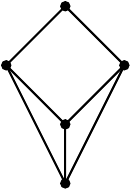}

	\vspace{\baselineskip}
	\vspace{\baselineskip}
	
	\glfig{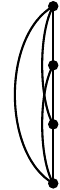}
	\hspace{0.5em}
	\glfig{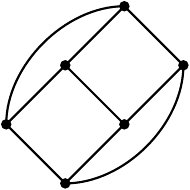}
	\hspace{0.5em}
	\glfig{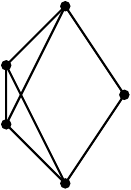}
	\hspace{0.5em}
	\glfig{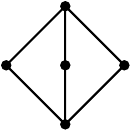}
	\hspace{0.5em}
	\glfig{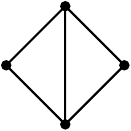}
\caption{The diagrams of the~10 {\genlatts} with~3 generators.}
\label{10_genlatts}
\end{figure}

Now we define the deletion and contraction operations for {\genlatts}.
These will be central concepts throughout.

\begin{definition}
\label{del-contr def}
	Let~$(L,G)$ be a {\genlatt} and let~$I\subseteq G$.
	The \emph{deletion of~$(L,G)$ by~$I$} is the {\genlatt}
	\begin{align*}
		(L,G)\wout I=\langle G\wout I|\zerohat\rangle.
	\end{align*}

	Let~$i_0=\bigjoin_{i\in I}i$.
	The \emph{contraction of~$(L,G)$ by~$I$} is the {\genlatt}
	\begin{align*}
		(L,G)/I=\langle \{g\join i_0:g\in G\}\wout\{i_0\}|i_0\rangle.
	\end{align*}
\end{definition}

For notational convenience we also define
the \emph{restriction of~$(L,G)$ to~$I$} to be the {\genlatt}
\begin{align*}
	(L,G)|_I=(L,G)\wout(G\wout I)=\langle I|\zerohat\rangle.
\end{align*}

For~$\ell\in L$ let~$G_{\le\ell}=\{g\in G:g\le\ell\}$. For~$I\subseteq G$
setting~$i_0=\bigjoin_{i\in I}i_0$, as above, the contraction~$(L,G)/I$
is the same as~$(L,G)/G_{\le i_0}$ since the join over~$I$ and~$G_{\le i_0}$ are
the same. Thus, contractions of~$(L,G)$ are indexed by~$L$.

\begin{definition}
	Given a {\genlatt} $(L,G)$ a \emph{minor} is any {\genlatt}
	obtained from $(L,G)$ via a sequence of deletion and contraction
	operations.
\end{definition}

We consider~$(L,G)$
to be a minor of itself.	
See~\Cref{fig square minors} for examples of minors of a {\genlatt}.

\begin{figure}
	\centering
	\includegraphics{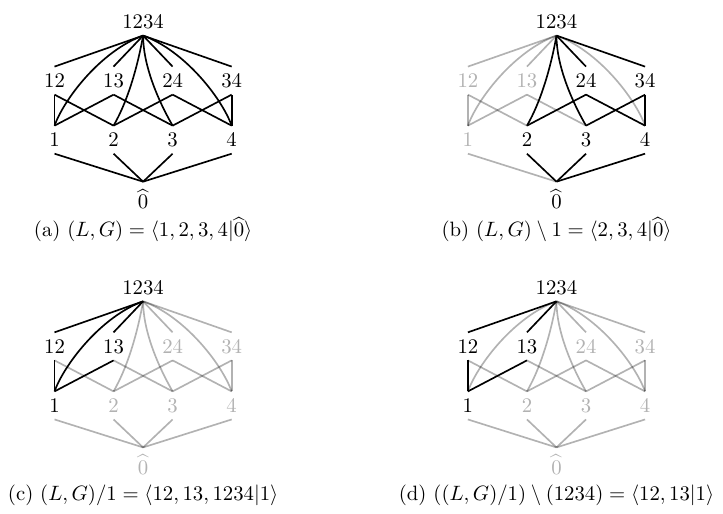}
	\caption{
		In (a) is the diagram of the face lattice of the square
		as a minimally generated lattice. In (b)-(d) are embedded diagrams
		of several minors.
	}
	\label{fig square minors}
\end{figure}

The following two lemmas which appear in
\cite{latticeMinorsAndPolymatroids}
will be used extensively
in the next section. We provide proofs here for completeness.
The use of \Cref{contr then del}
is somewhat ubiquitous and will be used without explicit reference.

\begin{lemma}[{\cite[Lemma 3.6]{latticeMinorsAndPolymatroids}}]
\label{contr then del}
Any minor of a {\genlatt}~$(L,G)$ may be expressed as the result of a
contraction followed by a deletion. Namely, a minor~$(K,H)$ of~$(L,G)$ may be expressed
as
\[(K,H)=((L,G)/G_{\le\zerohat_K})\wout\{g\join\zerohat_K:g\in G\wout G_{\le\zerohat_K},
\ g\join\zerohat_K\not\in H\}.\]
\end{lemma}

	\begin{proof}
		Let~$(K,H)$ be a minor of~$(L,G)$.
		By definition~$(K,H)$ may be expressed as the result of a sequence
		of contractions and deletions. That is, for some possibly
		empty sets of generators~$I_1,J_1,
		\dots,I_r,J_r$, that
		\[(K,H)=((\cdots(((L,G)/I_1)\setminus J_1)\cdots/I_r)\setminus J_r.\]
		For~$1\le j\le r$ let~$i_j$ be the join of all elements
		in~$I_j$. Set~$i_0=i_1\join\cdots\join i_r$.
		By definition of deletion and contraction, the minimal element~$\widehat{0}_K$
		of~$K$ is~$i_0$. Furthermore, the generators of~$(K,H)$ can each
		be expressed as~$g\join i_1\join\cdots\join i_r=g\join i_0$
		for some~$g\in G$. Thus each generator of~$(K,H)$ is
		a generator of~$(L,G)/i_0$, hence~$(K,H)=((L,G)/i_0)|_{H}$.
	\end{proof}

\begin{lemma}[{\cite[Lemma 3.7]{latticeMinorsAndPolymatroids}}]
\label{minor gen sets}
For any {\genlatt}~$(L,G)$ the minors are precisely {\genlatts} of the
form~$\langle\ell\join g_1,\dots,\ell\join g_k|\ell\rangle$
for~$\ell\in L$ and~$\{g_1,\ldots,g_k\}\subseteq G$ such that~$g_j\not\le\ell$
for~$1\le j\le k$.
\end{lemma}

	\begin{proof}
		Consider a minor~$(K,H)=((L,G)/I)|_J$ of~$(L,G)$,
		where~$I$ and~$J$ are sets of generators.
		Let~$\ell$ be the join of all elements of~$I$
		and let~$J=\{j_1,\dots,j_k\}$.
		By definition
		\[(K,H)=\langle\ell\join j_1,\dots,\ell\join j_k|\ell\rangle.\]
		Conversely, consider a
		{\genlatt}~$(K,H)=\langle\ell\join g_1,\dots,\ell\join g_k|\ell\rangle$
		for some~$\ell\in L$ and~$g_j\in G$ with~$g_j\not\le\ell$ for~$1\le j\le k$.
		The generators of the contraction~$(L,G)/\ell$ are all
		elements~$\ell\join g$ for~$g\in G$ with~$g\not\le\ell$.
		Thus~$\ell\join g_1,\dots,\ell\join g_k$ are generators of~$(L,G)/\ell$, so
		setting~$I=\{\ell\join g_1,\dots,\ell\join g_k\}$
		we have that~$(K,H)=((L,G)/\ell)|_I$.
	\end{proof}

\section{The minor poset}
\label{strong minor poset section}
In this section we introduce the minor poset of a {\genlatt}.
We then establish some basic results needed in \Cref{CW sphere section}.

\begin{definition}
	Let~$(L,G)$ be a {\genlatt}. The minor poset, denoted~$\M(L,G)$, is the
	poset consisting of minors of~$(L,G)$ adjoined with a
	minimal element $\emptyset$.
	The order relation is defined
	by~$(K_1,H_1)\le(K_2,H_2)$
	when~$(K_1,H_1)$ is a minor of~$(K_2,H_2)$.
\end{definition}

Figure~\ref{tam 3 minor poset fig} shows an example of a minor poset.
Below, we gather a few basic observations in a lemma.

\begin{lemma}
	For any {\genlatt} $(L,G)$ the following statements hold
	for the minor poset:
	\begin{itemize}
		\item[(i)]{The maximum of $\M(L,G)$ is $(L,G)$ itself.}
		\item[(ii)]{A lower interval $[\emptyset,(K,H)]$ in
			the minor poset $\M(L,G)$ is the minor poset
			$\M(K,H)$.}
		\item[(iii)]{The atoms of $\M(L,G)$ are the minors
			of $(L,G)$ with no generators, such minors
			correspond to the elements of $L$.}
		\item[(iv)]{The elements of $\M(L,G)$ that cover
			an atom are those with a single generator,
			such minors correspond to the edges of
			the diagram of $(L,G)$.}
	\end{itemize}
\end{lemma}

\begin{figure}
	\centering
	\includegraphics[height=0.45\textheight]{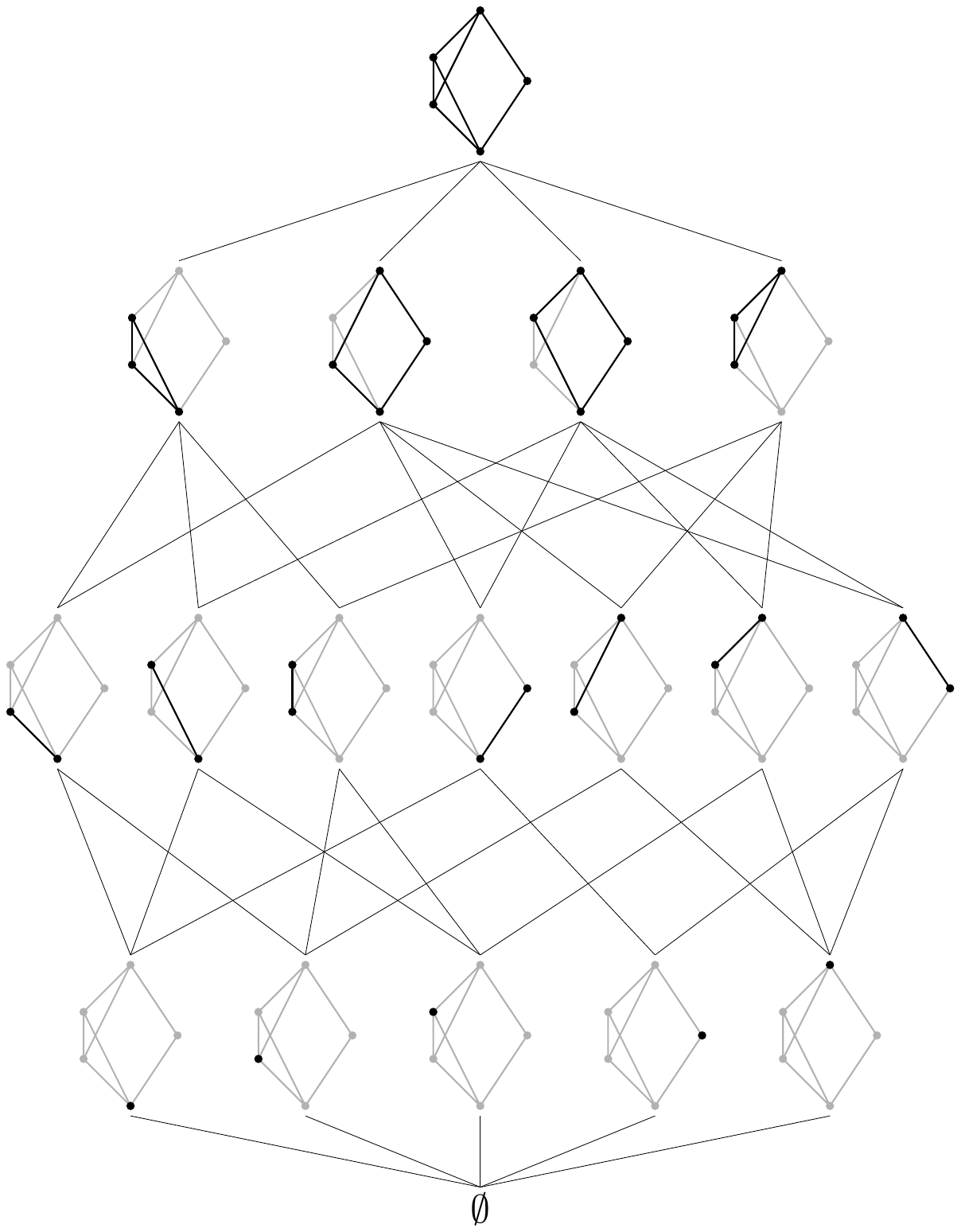}
	\caption{The minor poset of a {\genlatt} that has three generators.}
\label{tam 3 minor poset fig}
\end{figure}

Recall a ranked poset is said to be \emph{thin} if all length~2 intervals
are isomorphic to the Boolean algebra~$B_2$.
The following lemma, whose proof is somewhat technical,
is used to show minor posets are thin and graded.

\begin{lemma}
\label{rk 2 intervals}
	Let~$(L,G)$ be a {\genlatt} and let~$(K_1,H_1)$ and~$(K_2,H_2)$ be minors
	of~$(L,G)$ such that~$(K_1,H_1)<(K_2,H_2)$ in the minor poset~$\M(L,G)$.
	If~$\lvert{H_2}\rvert-\lvert{H_1}\rvert=2$
	then the interval~$[(K_1,H_1),(K_2,H_2)]$ of~$\M(L,G)$
	is isomorphic to the Boolean algebra~$B_2$.
\end{lemma}
	\begin{proof}
		The minor~$(K_1,H_1)$ may be presented as $((K_2,H_2)/I)\setminus J$ for
		some sets of generators~$I$ and~$J$. Proceed by
		considering the different possibilities for~$I$ and~$J$.
		When~$I$ is empty the set~$J$ must contain two elements,
		say~$j_1$ and~$j_2$. Then~$(K_1,H_1)=(K_2,H_2)\setminus\{j_1,j_2\}$
		and the open interval~$((K_1,H_1),(K_2,H_2))$ consists of the two
		minors~$(K_2,H_2)\setminus\{j_1\}$ and~$(K_2,H_2)\setminus\{j_2\}$.
		Since these two minors are incomparable the closed interval\[[(K_1,H_1),
		(K_2,H_2)]\] is isomorphic to~$B_2$.

		Now consider the case where~$J$ is empty. In this case~$I$
		may consist of either one element or two elements.
		First, suppose~$I=\{i_1,i_2\}$. 
		As a further subcase assume~$i_1\not<i_2$
		and~$i_1\not>i_2$. Since by
		assumption~$\lvert{H_2}\rvert-\lvert{H_1}\rvert=2$,
		every generator~$j$ of~$(K_1,H_1)$ corresponds to a unique generator~$i$
		of~$(K_2,H_2)$ such that~$i\join i_1\join i_2=j$. Due to this uniqueness
		no deletion of~$(K_2,H_2)$ has~$(K_1,H_1)$
		as a minor. If~$(K,H)=(K_2,H_2)\setminus\{j'\}$
		then~$j'\join i_1\join i_2$ is not
		a generator of~$(K,H)/\{i_1,i_2\}$ but is a generator of~$(K_1,H_1)$.
		By similar reasoning no contraction of~$(K_2,H_2)$
		other than~$(K_2,H_2)/\{i_1\}$ and~$(K_2,H_2)/\{i_2\}$
		(as well as~$(K_1,H_1)$ itself)
		has~$(K_1,H_1)$ as a minor.
		This establishes
		that the open interval~$((K_1,H_1),(K_2,H_2))$ consists solely of the
		minors~$(K_2,H_2)/\{i_1\}$ and~$(K_2,H_2)/\{i_2\}$, which are incomparable.

		Now return to the case~$I=\{i_1,i_2\}$ and
		suppose~$i_1<i_2$. In this case~$i_1\join i_2=i_2$
		so ~$(K_2,H_2)/\{i_1,i_2\}=(K_2,H_2)/\{i_2\}$.
		By a similar argument as used in the previous subcase,
		each generator~$j$ of~$(K_1,H_1)$ corresponds to a unique
		generator~$j'$ of~$(K_2,H_2)$ such that~$j'\join i_2=j$.
		Due to this uniqueness no deletion
		of~$(K_2,H_2)$ other than~$(K_2,H_2)\setminus\{i_1\}$
		contains~$(K_1,H_1)$ as a minor.
		Similarly, no contraction of~$(K_2,H_2)$
		with the exceptions of~$(K_2,H_2)/\{i_1\}$ and~$(K_1,H_1)$
		itself contain~$(K_1,H_1)$ as a minor. Thus, the open
		interval~$((K_1,H_1),(K_2,H_2))$ consists solely of the
		minors~$(K_2,H_2)/\{i_1\}$ and~$(K_2,H_2)\setminus\{i_1\}$.

		Now suppose~$I=\{i\}$ and~$J=\emptyset$. If there exists a generator~$i'\in H_2$
		such that~$i'<i$ in $K_2$ then we may take~$I=\{i,i'\}$ which falls under
		a previous case.
		Otherwise, there is  one
		generator~$j\in H_1$ that corresponds to two generators~$j_1,j_2\in H_2$,
		while all other generators in~$H_1$ correspond to a unique generator in~$H_2$.
		Thus, in this case\[((K_1,H_1),(K_2,H_2))=\{(K_2,H_2)\setminus\{j_1\},
		(K_2,H_2)\setminus\{j_2\}\}.\]

		What remains is the case where both~$I$
		and~$J$ are singletons. Suppose~$I=\{i\}$ and~$J=\{i\join j\}$
		for some generator~$j$ of~$(K_2,H_2)$.
		Once again since~$\lvert{H_2}\rvert-\lvert{H_1}\rvert=2$, every generator
		of~$(K_1,H_1)$ corresponds to a unique generator of~$(K_2,H_2)$.
		Hence, for all
		generators~$j_1,j_2$ of~$(K_2,H_2)$ the
		joins~$i\join j_1$ and~$i\join j_2$ are distinct,
		and the open interval~$((K_1,H_1),(K_2,H_2))$ consists
		of the minors~$(K_2,H_2)\setminus\{j\}$ and~$(K_2,H_2)/\{i\}$.
	\end{proof}

\begin{lemma}
\label{minor posets are graded and thin}
	For any {\genlatt}~$(L,G)$ the minor poset~$\M(L,G)$
	is graded by~$\rk(K,H)=\lvert{H}\rvert+1$ and is thin.
\end{lemma}

	\begin{proof}
		By Lemma~\ref{rk 2 intervals} whenever~$(K_1,H_1)\prec (K_2,H_2)$
		the difference~$\lvert{H_2}\rvert-\lvert{H_1}\rvert$ must be less than~2.
		Clearly this difference is positive, so it must be equal to~1.
		Furthermore, since every atom of~$\M(L,G)$ has zero generators, every saturated chain
		from the minimum~$\emptyset$ to a minor~$(K,H)$ must have the same
		length, namely~$\lvert{H}\rvert+1$.

		Since the minor poset~$\M(L,G)$ is graded by~$\rk(K,H)=\lvert{H}\rvert+1$,
		it follows from Lemma~\ref{rk 2 intervals} that the poset~$\M(L,G)$ is thin.
	\end{proof}

The following lemmas will assist in establishing the isomorphism types of minor posets
of minimally generated Boolean algebras and chains.

\begin{lemma}
\label{join-irr-fiber-lemma}
	Let $\phi:L\rightarrow K$ be a join-preserving map between finite lattices.
	If $i\in K$ is a join-irreducible then the fiber $\phi^{-1}(i)$ is either empty, the singleton $\{\zerohat_L\}$
	or contains a join-irreducible of $L$.
\end{lemma}

	\begin{proof}
		Take $\ell\in\phi^{-1}(i)\wout\{\zerohat_L\}$ and induct on the number of join-irreducibles
		below $\ell$. If $\ell\not\in\irr(L)$ then $\ell=\ell_1\join\ell_2$
		for some elements such that $\ell\ne\ell_1,\ell_2$. Since $\phi$ is join-preserving
		we have $\phi(\ell_1)\join\phi(\ell_2)=i$. Since $i$ is join-irreducible either
		$\phi(\ell_1)=i$ or $\phi(\ell_2)=i$. The induction hypothesis implies $\phi^{-1}(i)$
		contains a join-irreducible of $L$.
	\end{proof}
\begin{lemma}
\label{when minor gequal is containment}
	Let~$(L,G)$ be a {\genlatt} and let~$(K,H)$ and~$(M,\irr(M))$ be minors of~$(L,G)$.
	If~$\irr(M)\subseteq\irr([\zerohat_M,\onehat_L])$
	and~$M\subseteq K$ then~$(M,\irr(M))\le(K,H)$.
\end{lemma}

	\begin{proof}
		By assumption,~$\irr(M)\subseteq\irr([\zerohat_M,\onehat_L])\cap K$.
		An element~$x\in\irr([\zerohat_M,\onehat_L])\cap K$ is an
		element of~$[\zerohat_M,\onehat_L]\cap K$ that
		covers only one element that is greater than or equal to~$\zerohat_M$.
		Such an element~$x$ is also an element of~$\irr([\zerohat_M,\onehat_K])$
		so~$\irr(M)\subseteq\irr([\zerohat_M,\onehat_K])\cap K$.
		Using the inclusion map~$[\zerohat_M,\onehat_K]\cap K\rightarrow[\zerohat_M,\onehat_K]$
		\Cref{join-irr-fiber-lemma} implies
		\[\irr([\zerohat_M,\onehat_K])\cap K=\irr([\zerohat_M,\onehat_K]\cap K).\]
		Thus, we have established
		\[\irr(M)\subseteq\irr([\zerohat_M,\onehat_K]\cap K).\]

		Consider the map~$k\mapsto\zerohat_M\join k$ which is a join-preserving
		surjection from~$K$ onto the interval~$[\zerohat_M,\onehat_K]\cap K$ of~$K$.
		The image of~$\irr(K)$ includes~$\irr([\zerohat_M,\onehat_K]\cap K)$,
		hence the image of~$H$ includes~$\irr(M)$. This establishes that
		the minor~$(M,\irr(M))$ is a deletion of~$(K,H)/H_{\le\zerohat_M}$.
	\end{proof}

\begin{proposition}
\label{M(B_n) is cube}
	The minor poset~$\M(B_n,\irr(B_n))$ of the Boolean algebra
	is isomorphic to the face lattice,~$Q_n$, of
	the~$n$-dimensional cube, that is,
	\[\M(B_n,\irr(B_n))\cong Q_n\]
\end{proposition}

	\begin{proof}
		Recall, the face lattice of the~$n$-dimensional cube
		is isomorphic to the poset of intervals of~$B_n$ with
		a minimum~$\widehat{0}$ appended
		\cite[Chapter~3,~Exercise~177]{stanley}.
		It is straightforward to see from Lemma~\ref{minor gen sets}
		that the minors of~$(B_n,\irr(B_n))$
		are exactly the intervals of~$B_n$. Indeed a minor of
		the form~$\langle X\cup\{i_1\},\dots,X\cup\{i_r\}|X\rangle$
		is the interval~$[X,X\cup\{i_1,\dots,i_r\}]$ with minimal generating set.
		Proposition~\ref{when minor gequal is containment} implies the order relations are
		the same.
	\end{proof}

Let~$C_n$ denote the length~$n$ chain~$\{0<1<\dots<n\}$.

\begin{proposition}
	The minor poset~$\M(C_n,\irr(C_n))$ of the length~$n$ chain~$C_n$ is isomorphic to the rank~$n+1$
	Boolean algebra, that is,
	\[\M(C_n,\irr(C_n))\cong B_{n+1}.\]
\end{proposition}

	\begin{proof}
		In a chain every element except the minimum
		is join-irreducible
		and must be a generator. As a result, every nonempty subset of the lattice~$C_n$
		is a minor; a subset~$\{i_1<\dots<i_r\}$ of $\{0,\dots,n\}$ is the
		minor~$\langle i_1\join i_2,\dots,i_1\join i_r|i_1\rangle$.
		Proposition~\ref{when minor gequal is containment} implies the minors are ordered
		by inclusion.
	\end{proof}

The following lemma characterizes the join operation in the minor poset for any {\genlatt}.

\begin{lemma}
\label{minor poset joins}
	Let~$(L,G)$ be a {\genlatt} and let\[\theta:(B_G,\irr(B_G))\rightarrow(L,G)/G_{\le\ell_0}\]
	be the canonical strong map. Let~$(K_1,H_1)$ and~$(K_2,H_2)$ be minors
	of~$(L,G)$ and set \[\ell_0=\widehat{0}_{K_1}\wedge\widehat{0}_{K_2}.\]
	The join of~$(K_1,H_1)$ and~$(K_2,H_2)$ in the minor poset~$\M(L,G)$ exists
	if and only if the following three conditions hold:
	\begin{enumerate}
		\item[\eqlabel{join condition 1}]{The
		fibers~$\theta^{-1}(\widehat{0}_{K_1})$ and~$\theta^{-1}(\widehat{0}_{K_2})$
		each have a minimum, say,~$X_1$ and~$X_2$
		respectively.}
		
		\item[\eqlabel{join condition 2}]{For each generator~$h\in H_1$ there is a unique
		generator~$g$ of~$(L,G)/G_{\le\ell_0}$ with~$g\join\widehat{0}_{K_1}=h$, and similarly
		for~$(K_2,H_2)$.}
		
		\item[\eqlabel{join condition 3}]{Any minor~$(K,H)$ of~$(L,G)$ such that~$(K,H)\ge(K_1,H_1)$
		and~$(K,H)\ge (K_2,H_2)$
		contains the element~$\ell_0$.}
	\end{enumerate}

	Let~$I$ be the set
	\begin{align}I=&\{\theta(\{x\}):x\in X_1\cup X_2\}\notag\\
		\cup&\{g\join\ell_0:g\in G\text{ and }g\join\widehat{0}_{K_1}\in H_1\}\notag\\
		\cup&\{g\join\ell_0:g\in G\text{ and }g\join\widehat{0}_{K_2}\in H_2\}.
		\label{join gen set}
	\end{align}
	If the above conditions are satisfied then
	\[(K_1,H_1)\join(K_2,H_2) = ((L,G)/G_{\le\ell_0})|_I.\]
\end{lemma}

	\begin{proof}
		Assume the join
			\[(K,H)=(K_1,H_1)\join (K_2,H_2)\]
		exists in~$\M(L,G)$.
		We claim~$\zerohat_K=\ell_0$.
		To prove this, first observe since~$(K,H)\ge(K_i,H_i)$ for~$i=1,2$,
		by taking minimal elements we see~$\zerohat_K\le\zerohat_{K_i}$
		hence~$\zerohat_K\le\ell_0$. To show the opposite inequality,
		observe~$(K_i,H_i)\le(L,G)/G_{\le\ell_0}$ for~$i=1,2$,
		hence~$(K,H)\le(L,G)/G_{\le\ell_0}$; this implies~$\zerohat_K
		\ge\ell_0$, hence~$\zerohat_K=\ell_0$. In particular,
		we have established that Condition \eqref{join condition 3} holds.
		
		The join~$(K_1,H_1)\join (K_2,H_2)$
		is a deletion of the minor~$(L,G)/G_{\le\ell_0}$,
		say,
		\[(K_1,H_1)\join (K_2,H_2)=((L,G)/G_{\le\ell_0})|_J\]
		for some set~$J$ of generators.
		The set~$J$ satisfies
		the following two conditions for~$i=1,2$:
			\begin{enumerate}
				\item{$\widehat{0}_{K_i}$ can be expressed as a join of
				elements in~$J$}
				\item{The set~$H_i$ is included in the
				set~$\{j\join\widehat{0}_{K_i}:j\in J\}$}
			\end{enumerate}
		Furthermore,~$J$ is the unique minimal set satisfying Conditions 1 and 2
		above, for if~$J'$ is another set with these properties
		then~$((L,G)/G_{\le\ell_0})|_{J'}$ is greater than or equal to
		both~$(K_1,H_1)$ and~$(K_2,H_2)$;
		thus
		\[((L,G)/G_{\le\ell_0})|_J\le((L,G)/G_{\le\ell_0})|_{J'}\]
		and~$J\subseteq J'$.
		The existence of the set~$J$ implies
		Conditions \eqref{join condition 1} and \eqref{join condition 2} hold.

		Now suppose Conditions \eqref{join condition 1} through
		\eqref{join condition 3}
		are satisfied. Let~$I$ be the
		set defined in \Cref{join gen set}.
		Consider a minor~$(K',H')$ of~$(L,G)$ that satisfies
		\[(K',H')\ge(K_1,H_1)
		\phantom{???}\text{and}\phantom{???}
		(K',H')\ge(K_2,H_2),\]
		we will show~$(K',H')\ge((L,G)/G_{\le\ell_0})|_I$.
		By Condition~\eqref{join condition 3}~$\ell_0\in K'$,
		so the inequality
		\[(K',H')/H'_{\le\ell_0}\ge(K_i,H_i)\]
		is satisfied for~$i=1,2$. We assume~$\zerohat_{K'}=\ell_0$
		by replacing~$(K',H')$ with~$(K',H')/H'_{\le\ell_0}$.

		By assumption,
		$(K',H')\ge(K_i,H_i)$
		for~$i=1,2$, hence,
			\[H_i\subseteq\{h\join\zerohat_{K_i}:h\in H'\}\]
		and~$\zerohat_{K_i}\in K'$. The uniqueness in
		Condition~\eqref{join condition 1} and~\eqref{join condition 2}
		imply~$H'$ includes the set~$I$ defined in \Cref{join gen set}.
		Thus,~$(K',H')\ge((L,G)/G_{\le\ell_0})|_I$.
	\end{proof}

In the remainder of this section we show two results concerning how operations affect the
minor poset. First, we establish the following lemma which allows us to induce
order-preserving maps on minor posets from strong maps.
Then we examine minor posets for Cartesian products of {\genlatts} and the
effect of adjoining a new maximum to a {\genlatt}.

\begin{lemma}
\label{induced map}
	Let~$f:(L,G)\rightarrow(K,H)$ be a strong map between {\genlatts}.
	The induced map~$F:\M(L,G)\rightarrow\M(K,H)$
	defined by
	\[F(\langle I|z\rangle)=\langle f(I)|f(z)\rangle\]
	is order-preserving. Furthermore, if~$f$ is surjective so is~$F$.
\end{lemma}

	\begin{proof}
		Let~$(L_1,G_1)$ and~$(L_2,G_2)$ be minors of~$(L,G)$
		such that~$(L_1,G_1)\le(L_2,G_2)$.
		Set~$(K_1,H_1)=F(L_1,G_1)$ and~$(K_2,H_2)=F(L_2,G_2)$.
		Since~$(L_1,G_1)$ is a minor of~$(L_2,G_2)$
		the generating
		set~$G_1$ is a subset of the set~$\{g\join\widehat{0}_{L_1}:
		g\in G_2\}\setminus\{\widehat{0}_{L_1}\}$. Applying the map~$F$,
		the set~$H_1$ is a subset of the set
		\[
		\{f(g)\join f(\widehat{0}_{L_1}):g\in G_2\}\setminus\{f(\widehat{0}_{L_1})\}
		=\{h\join\widehat{0}_{K_1}:h\in H_2\}\setminus\{\widehat{0}_{K_1}\}.
		\]
		This establishes that if~$\widehat{0}_{K_1}$ is an element
		of~$K_2$ then~$(K_1,H_1)$ is a deletion
		of
			\[(K_2,H_2)/\{h\in H_2:h\le\widehat{0}_{K_1}\}\]
		hence~$(K_1,H_1)\le(K_2,H_2)$.
		For~$\widehat{0}_{K_1}$ to be contained in~$K_2$
		the element~$\widehat{0}_{K_1}$ must be able to
		be expressed as a join of elements from~$H_2$. The
		element~$\widehat{0}_{L_1}$ can be expressed as a
		join~$\widehat{0}_{L_1}=g_1\join\cdots\join g_r$
		for some generators~$g_i\in G_2$ for~$i=1,\dots,r$.
		Applying the map~$f$ yields~$\widehat{0}_{K_1}=f(g_1)\join
		\cdots\join f(g_r)$; the elements~$f(g_i)$ are contained
		in~$H_2\cup\{\widehat{0}_{K_2}\}$.
		Removing any elements~$f(g_i)=\zerohat_{K_2}$ gives the desired expression
		of~$\zerohat_{K_1}$.
		Therefore~$(K_1,H_1)\le(K_2,H_2)$ and the induced
		map~$F$ is order-preserving.

		The surjectivity claim follows directly from \Cref{minor gen sets}.
	\end{proof}

We next examine the Cartesian product of {\genlatts}.
Given two {\genlatts}~$(L,G)$ and~$(K,H)$ define the \emph{Cartesian
product}~$(L,G)\times(K,H)$ to be the {\genlatt} 
	\[(L,G)\times(K,H)=(L\times K,(G\times\{\widehat{0}_K\})\cup
	(\{\widehat{0}_L\}\times K)).\]
This operation corresponds to the diamond product of minor posets.
Recall that for two posets~$P$ and~$Q$ the \emph{diamond
product}~$P\diamond Q$ is defined to be the poset
	\[((P\setminus\{\widehat{0}_P\})\times
	(Q\setminus\{\widehat{0}_Q\}))\cup\{\widehat{0}\}\]
in which~$\zerohat$ is a new minimum.
The pyramid and prism operations on a poset $P$ are defined
as~$\pyr(P)=P\times B_1$ and~$\prism(P)=P\diamond B_2$.
We define~$\pyr$ on {\genlatts}
in the same manner, that is,~$\pyr(L,G)=(L,G)\times(B_1,\irr(B_1))$.

The following proposition shows Cartesian products correspond to diamond products.
This is a generalization of
the fact that minor poset of a rank~$n$ Boolean algebra is the
face lattice of an $n$-dimensional cube.
\begin{proposition}
\label{coproduct to diamond product}
	For any {\genlatts}~$(L,G)$ and~$(K,H)$ we have
	\[\M((L,G)\times(K,H))\cong\M(L,G)\diamond\M(K,H).\]
	In particular,~$\M(\pyr(L,G))\cong\prism(\M(L,G))$.
\end{proposition}

	\begin{proof}
		Let~$\pi_1:(L,G)\times(K,H)\rightarrow(L,G)$
		and~$\pi_2:(L,G)\times(K,H)\rightarrow(K,H)$ be the obvious projection maps.
		By Lemma~\ref{induced map}
		these induce order-preserving maps~$\overline{\pi}_1$
		and~$\overline{\pi}_2$ between the minor posets.
		Neither of these maps sends a minor to the minimum~$\emptyset$,
		so we have an order-preserving
		map~$\phi:\M((L,G)\times(K,H))\rightarrow\M(L,G)\diamond
		\M(K,H)$ defined on minors by~$\phi(M,I)=(\overline{\pi}_1(M,I),
		\overline{\pi}_2(M,I))$.

		The inverse of~$\phi$ is the map~$\psi$ defined for~$((L',G'),(K',H'))$
		in~$(\M(L,G)\diamond\M(K,H))\setminus\{\emptyset\}$ by
		\[\psi((L',G'),(K',H'))=\langle (G'\times\{\widehat{0}_{K'}\})\cup
		(\{\widehat{0}_{L'}\}\times H')|(\widehat{0}_{L'},\widehat{0}_{K'})\rangle.\]
		To see that the image under~$\psi$ is indeed a minor,
		let~$G''\subseteq G$ be such that
		\[G'=\{g\join\widehat{0}_{L'}:g\in G''\}.\]
		and similarly define~$H''$. The generating set of the image under~$\psi$
		can be described as
		\[
		\{(g,\widehat{0}_K)\join(\widehat{0}_{L'},\widehat{0}_{K'}):g\in G''\}
		\cup\{(\widehat{0}_L,h)\join(\widehat{0}_{L'},\widehat{0}_{K'}):h\in H''\},
		\]
		which is the generating set of a minor by Lemma~\ref{minor gen sets}.

		Observe for~$I\subseteq G'$ that
		\[\psi((L',G')\setminus I,(K',H'))
		=\psi((L',G'),(K',H'))\setminus(I\times\{\widehat{0}_{K'}\}),\]
		and~\[\psi((L',G')/I,(K',H'))
		=\psi((L',G'),(K',H'))/(I\times\{\widehat{0}_{K'}\}.\]
		A similar statement holds for deletions and contractions
		of~$(K',H')$. Thus, the map~$\psi$ is order-preserving.
	\end{proof}

We now consider the minor poset of the result of adjoining a new maximum
to a {\genlatt}.
Given a{ \genlatt}~$(L,G)$ let~$\widehat{L}$ denote the lattice obtained
from~$L$ by adjoining a new maximal element~$m>\widehat{1}_L$
and let~$\widehat{G}=G\cup\{m\}$. 

The following proposition may be thought of as a generalization of the fact
that the minor poset of a chain is a Boolean algebra.

\begin{proposition}
	For any {\genlatt}~$(L,G)$ we have
	\[\M(\widehat{L},\widehat{G})\cong\pyr(\M(L,G)).\]
\end{proposition}

	\begin{proof}
		Let~$m$ be the maximum of~$\widehat{L}$.
		Observe, since~$m$ is join-irreducible and the maximum
		of~$\widehat{L}$, for any
		minor~$(K,H)\ne\langle\emptyset|m\rangle$ of~$(\widehat{L},\widehat{G})$,
		both~$(K\cup\{m\},H\cup\{m\})$
		and~$(K\setminus\{m\},H\setminus\{m\})$
		are minors of~$(L,G)$.
		Define a map~$\phi:\M(\widehat{L},\widehat{G})\rightarrow\pyr(\M(L,G))$
		by setting~$\phi(\emptyset)=(\emptyset,\widehat{0})$, and for minors~$(K,H)$
		of~$(L,G)$ setting
		\[
			\phi(K,H)=\begin{cases}
					(K,H,\widehat{0})&\text{if }m\not\in K,\\
					(K\setminus\{m\},H\setminus\{m\},\widehat{1})&
					\text{if }m\in H,\\
					(\emptyset,\widehat{1})&\text{if }K=\{m\},
					\ H=\emptyset.\\
				\end{cases}
		\]
		The map~$\phi$ is order-preserving.
		Define~$\psi:\pyr(\M(L,G))\rightarrow\M(\widehat{L},\widehat{G})$
		by setting~$\psi(\emptyset,\widehat{0})=\emptyset$
		and~$\psi(\emptyset,\widehat{1})=\langle\emptyset|m\rangle$
		and setting
		\[\psi(K,H,\varepsilon)=\begin{cases}
		(K,H)&\text{if }\varepsilon=\widehat{0},\\
		(K\cup\{m\},H\cup\{m\})&\text{if }\varepsilon=\widehat{1}.
		\end{cases}
		\]
		Clearly the map~$\psi$ is order-preserving and is the inverse of~$\phi$.
	\end{proof}

\section{Factoring strong surjections}
\label{factoring section}

In this section we provide a process to factor any strong surjection
into strong surjections that only identify two elements. This factorization
is used to simplify the zipping construction in \Cref{CW sphere section}.
One may consider this factorization as a polymatroid analogue of the factorization into
elementary strong maps via Higgs lifts from \cite{higgs}.

\begin{definition}
	Let~$\mathcal{E}(L,G)$ denote the poset consisting of edges of the diagram
	of a {\genlatt}~$(L,G)$, that is, pairs~$\{\ell,\ell\join g\}$ for~$\ell\in L$
	and~$g\in G$ such that~$\ell\join g\ne\ell$.
	The order relation of~$\mathcal{E}(L,G)$ is defined
	by~$\{\ell,\ell\join g\}\le\{a\join\ell,a\join\ell\join g\}$
	for~$a\in L$.
\end{definition}

\begin{definition}
\label{conn equiv rel def}
	Given an equivalence relation~$\phi$ on a {\genlatt}~$(L,G)$,
	an \emph{edge} of~$\phi$ is an edge~$\{\ell,\ell\join g\}$ of~$(L,G)$
	such that~$\ell\equiv\ell\join g(\phi)$.

	Let~$\mathcal{E}(\phi)$
	denote the set of edges of~$\phi$.
	The equivalence relation~$\phi$ is said to be
	\emph{connected} if for all~$a\equiv b(\phi)$ there is a
	sequence~$a=c_0,c_1,\dots,c_k=b$ such that for~$1\le i\le k$
	the pair~$\{c_{i-1},c_i\}$ is an edge of~$\phi$.
\end{definition}

Note that a connected relation is determined by its set of edges.

\begin{lemma}
\label{join relations and lift order ideals}
Let~$\phi$ be an equivalence relation on a {\genlatt}~$(L,G)$.
The relation~$\phi$ is join-preserving if and only if the following two conditions
hold.
	\begin{itemize}
		\item[(i)]{$\phi$ is connected.}
		\item[(ii)]{The set
		of edges~$\mathcal{E}(\phi)$ forms an upper order ideal in the
		edge poset~$\mathcal{E}(L,G)$.}
	\end{itemize}
\end{lemma}

	\begin{proof}
		Let~$\phi$ be a join preserving equivalence relation on~$(L,G)$.
		To show~$\phi$ is connected let~$a,b\in(L,G)$
		with~$a\equiv b(\phi)$. Then, the congruence~$a\equiv a\join b\equiv b(\phi)$
		is satisfied.
		Choose some sequence~$g_1,\dots,g_r$ of generators
		of~$(L,G)$ such that
			\[a<a\join g_1<\dots<a\join (g_1\join\cdots\join g_r)=a\join b.\]
		Each term of the sequence must be
		congruent to~$a$ since said terms lie in the interval~$[a,a\join b]$.
		Thus, each pair of subsequent terms in the sequence is an edge
		of~$\phi$. There exists a similarly defined sequence
		from~$b$ to~$a\join b$. Concatenating
		these two sequences gives a sequence from~$a$ to~$b$ consisting
		of edges of~$\phi$. Therefore~$\phi$ is connected.
		
		In order to show the set~$\mathcal{E}(\phi)$ forms an upper
		order ideal of the poset~$\mathcal{E}(L,G)$,
		let~$\{a,b\}$ be an edge of~$\phi$ and let~$\ell\in L$.
		Since~$a\equiv b(\phi)$ and~$\phi$ is
		join-preserving~$a\join\ell\equiv b\join\ell
		(\phi)$. Hence, if~$a\join\ell\ne b\join\ell$ then~$\{a\join\ell,
		b\join\ell\}$ is an
		edge of~$\phi$. So~$\mathcal{E}(\phi)$
		is an upper order ideal of~$\mathcal{E}(L,G)$.

		Conversely, consider a connected equivalence relation~$\phi$ on~$L$
		such that the set of edges~$\mathcal{E}(\phi)$
		forms an upper order ideal in the poset~$\mathcal{E}(L,G)$ of edges
		of~$(L,G)$. Let~$a,b\in L$ such that~$a\equiv b(\phi)$.
		Since~$\phi$ is connected there is a
		sequence~$a=c_0,c_1,\dots,c_k=b$ such that for~$1\le i\le k$
		the pair~$\{c_{i-1},c_i\}$ is an edge of~$\phi$.
		Taking the join with any~$\ell\in L$ results in a
		sequence~$a\join\ell=c_0\join\ell,c_1\join\ell,\dots,
		c_k\join\ell=b\join\ell$. For~$1\le i\le k$
		either~$c_{i-1}\join\ell=c_i\join\ell$
		or~$\{c_{i-1}\join\ell,c_i\join\ell\}$ is an edge of~$L$.
		In the second case since~$\mathcal{E}(\phi)$ is an upper order ideal,
		the edge~$\{c_{i-1},c_i\}\in\mathcal{E}(\phi)$ implies
		that~$\{c_{i-1}\join\ell,c_i\join\ell\}\in\mathcal{E}(\phi)$.
		Removing repeated terms results in a sequence
		from~$a\join\ell$ to~$b\join\ell$ consisting of
		edges of~$\phi$. Thus~$a\join\ell\equiv b\join\ell(\phi)$,
		hence~$\phi$ is a join preserving equivalence relation.
	\end{proof}

\begin{lemma}
\label{irreducible strong maps}
	Let~$(L,G)$ and~$(K,H)$ be {\genlatts} and
	let~$f:(L,G)\rightarrow(K,H)$
	be a strong map. If the map~$f$ has
	a single nontrivial fiber~$\{x,y\}$ in~$L$,
	then the element~$x$ is only covered by~$y$ or vice versa.
\end{lemma}

	\begin{proof}
		Since~$f(x)=f(y)$ it must be that~$f(x\join y)=f(x)\join f(y)
		=f(x)$; and since~$f^{-1}(x)=\{x,y\}$
		the element~$x\join y$ is either~$x$ or~$y$. Without
		loss of generality we may assume~$x<y$. In fact~$x\prec y$,
		since if~$x\le z \le y$ then~$f(x)\le f(z)\le f(y)=f(x)$
		so~$f(z)=f(x)$; hence,~$z$ must be equal to~$x$ or to~$y$.

		Suppose~$z>x$, then~$f(z)=f(x\join z)=f(y\join z)$.
		The map~$f$
		is invertible when restricted to~$L\setminus\{x\}$, so this
		implies~$y\join z = z$ hence~$y\le z$.
		Therefore~$x$ is only covered by~$y$.
	\end{proof}

\begin{lemma}
\label{strong map factorization}
	Let~$(L,G)$ and~$(K,H)$ be {\genlatts}. Any strong
	surjection~$f:(L,G)\rightarrow(K,H)$ can be factored
	as~$f=f_r\circ\dots\circ f_1$ where each map~$f_i$
	is a strong surjection that identifies only two elements.
\end{lemma}

	\begin{proof}
		Consider the map~$f$ as an equivalence relation on~$L$
		defined by~$a\equiv b(f)$ when~$f(a)=f(b)$.
		By Lemma~\ref{join relations and lift order ideals} this equivalence
		relation is connected and
		the edges form an upper order ideal in the
		poset~$\mathcal{E}(L,G)$ of edges of~$(L,G)$.
		Choose a linear extension of~$\mathcal{E}(L,G)$. Then order
		the edges of~$f$ as~$\{x_1,y_1\},\dots,\{x_r,y_r\}$
		by the chosen linear extension.
		Define equivalence relations~$f_i$ for~$1\le i\le r$ by
		letting~$f_i$ be the transitive closure of the relation
		defined by~$x_j\equiv y_j$ for~$1\le j\le i$.
		By definition~$f_i$ is connected and the edges of~$f_i$
		form an upper order ideal in~$\mathcal{E}(L,G)$
		so~$f_i$ is join-preserving. Thus, each relation~$f_i$
		defines a {\genlatt}, namely the quotient~$f_i(L,G)$.
		Furthermore, for each~$i$
		either~$f_{i+1}(L,G)=f_i(L,G)$
		or~$f_{i+1}(L,G)$ is obtained from~$f_i(L,G)$
		by identifying two elements, namely~$f_i(x_{i+1})$ and~$f_i(y_{i+1})$.
		Thus, the strong map~$f$ factors as the product~$f_r\circ\dots\circ f_1$
		of the quotient maps. Removing any instances of
		identity maps gives the desired factorization.
	\end{proof}

\Cref{map factorization fig} depicts an example of the factorization from \Cref{strong map factorization}.

\begin{figure}
\centering
\includegraphics[width=0.9\textwidth]{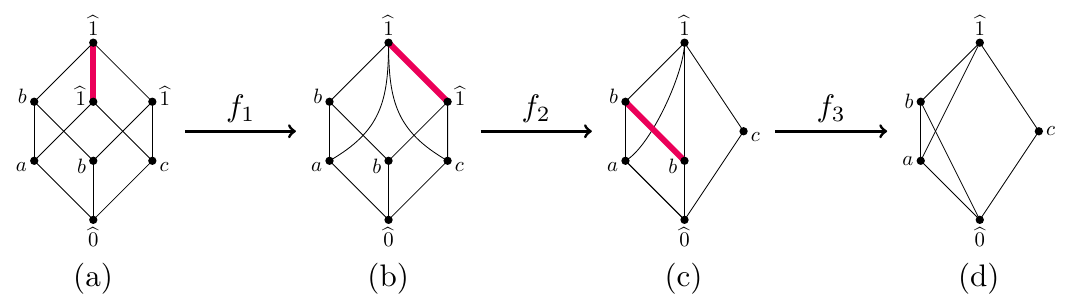}
\caption{An example of the factorization from \Cref{strong map factorization}. The labels
indicate the images in the final lattice (d) and the thicker red edges indicate the two elements
identified by the maps~$f_i$.}
\label{map factorization fig}
\end{figure}

\section{Minor posets as CW spheres}
\label{CW sphere section}

In this section we prove the main theorem, which is a zipping construction for minor
posets. 
For details on the
zipping operation, see the \hyperlink{zipping-appendix}{appendix}.
To begin, we first characterize fibers of maps between minor posets induced
by strong maps from the factorization of \Cref{irreducible strong maps}.

Recall from Lemma~\ref{induced map} that a strong map~$f:(L,G)\rightarrow(K,H)$
induces an order-preserving map~$F:\M(L,G)\rightarrow\M(K,H)$ defined
by
\[F(\langle I|z\rangle)=\langle f(I)|f(z)\rangle.\]

\begin{lemma}
\label{irreducible minor maps}
	Let~$f:(L,G)\rightarrow(K,H)$ be a strong map between {\genlatts}.
	If the map~$f$ has a single nontrivial fiber~$x\prec y$
	then the induced map~$F:\M(L,G)\rightarrow\M(K,H)$
	has nontrivial fibers of the form
	\[\{(M,I),(M_x,I_x),(M_y,I_y)\},\]
	where
	\begin{align}
		x,y&\in I\cup\{\widehat{0}_M\},
			\label{irr minor map fiber 1}\\
		(M_x,I_x)&=(M,I)\setminus\{y\},
			\label{irr minor map fiber 2}\\
		(M_y,I_y)&=\begin{cases}(M,I)\setminus\{x\}&\text{if }x\in I,\\
		(M,I)/\{y\}&\text{if }x=\widehat{0}_M.\end{cases}
			\label{irr minor map fiber 3}
	\end{align}
\end{lemma}

	\begin{proof}
		Clearly such a triple~$\{(M,I),(M_x,I_x),(M_y,I_y)\}$ intersects a single
		fiber of~$F$, and
		given any minor~$(M,I)$ with~$x,y\in I\cup\{\widehat{0}_M\}$
		there exists such a triple. Furthermore, the fiber containing~$(M,I)$
		is precisely the set~$\{(M,I),(M_x,I_x),(M_y,I_y)\}$.

		Now suppose~$(M_1,I_1)$ and~$(M_2,I_2)$
		are minors
		of~$(L,G)$ in the same fiber of~$F$ such that neither~$I_1\cup
		\{\widehat{0}_{M_1}\}$ nor~$I_2\cup\{\widehat{0}_{M_2}\}$ contains
		both~$x$ and~$y$. We must show there is a minor~$(M,I)$ such
		that the triple~$\{(M,I),(M_1,I_1),(M_2,I_2)\}$ is of the form
		described by Conditions~\eqref{irr minor map fiber 1}
		through~\eqref{irr minor map fiber 3}.
		By assumption~$f(\widehat{0}_{M_1})=f(\widehat{0}_{M_2})$.
		Either~$\widehat{0}_{M_1}=\widehat{0}_{M_2}$ or one is~$x$ and the other~$y$.
		Consider the case where~$\widehat{0}_{M_1}=\widehat{0}_{M_2}$.
		The equality~$F(M_1,I_1)=F(M_2,I_2)$ implies~$f(I_1)=f(I_2)$;
		hence, these sets differ by
		exchanging~$x$ and~$y$.
		There is a minor~$(M,I)$ of~$(L,G)$
		with~$I=I_1\cup I_2$ and~$\widehat{0}_M=\widehat{0}_{M_1}$.
		Furthermore, this minor satisfies
			\[F(M,I)=F(M_1,I_1)=F(M_2,I_2)\]
		and~$(M_1,I_1)$ and~$(M_2,I_2)$ are
		the minors~$(M,I)\setminus\{x\}$ and~$(M,I)\setminus\{y\}$.

		Now consider the case where~$\widehat{0}_{M_1}\ne\widehat{0}_{M_2}$.
		Since~$f(\widehat{0}_{M_1})=f(\widehat{0}_{M_2})$ one must
		be~$x$ and the other~$y$; say~$\widehat{0}_{M_1}=x$.
		Since~$x\prec y$ the element~$y$ is a generator of
		the contraction~$(L,G)/G_{\le x}$. Consequently, there is a minor~$(M,I)$
		of~$(L,G)$ with~$\widehat{0}_M=x$ and~$I=I_1\cup\{y\}$.
		Furthermore,~$(M_1,I_1)=(M,I)\setminus\{y\}$ and~$(M_2,I_2)=(M,I)/\{y\}$.
	\end{proof}

\begin{theorem}
\label{zip construction}
	Let~$(L,G)$ and~$(K,H)$ be {\genlatts} such that there is a strong
	surjection from~$(L,G)$ onto~$(K,H)$. The minor poset~$\M(K,H)$
	can be obtained from~$\M(L,G)$ via a sequence of zipping operations.
\end{theorem}

\begin{example}
\Cref{zip constr fig} depicts an example of the zipping construction associated to the
strong surjection whose factorization is depicted in \Cref{map factorization fig}. The
top row of CW complexes have face posets that, up to isomorphism, are the minor posets
of the lattices in \Cref{map factorization fig}. In that same figure
the maps~$F_i$ are induced by the
strong maps~$f_i$. The cells in red form the zippers that are zipped
to form the next CW complex (in alphabetic order (a) to (f)). Note, the two complexes~(c) and~(e)
are not minor posets. These posets are intermediary posets in the zipping construction.
\end{example}

In general, given a strong surjection~$f$ that identifies only two elements $x$ and $y$, the associated
minor poset zipping sequence is to zip the edge between $x$ and $y$
and then proceed upward in ranks zipping any newly created zippers.

\begin{figure}
\centering
\includegraphics[width=0.9\textwidth]{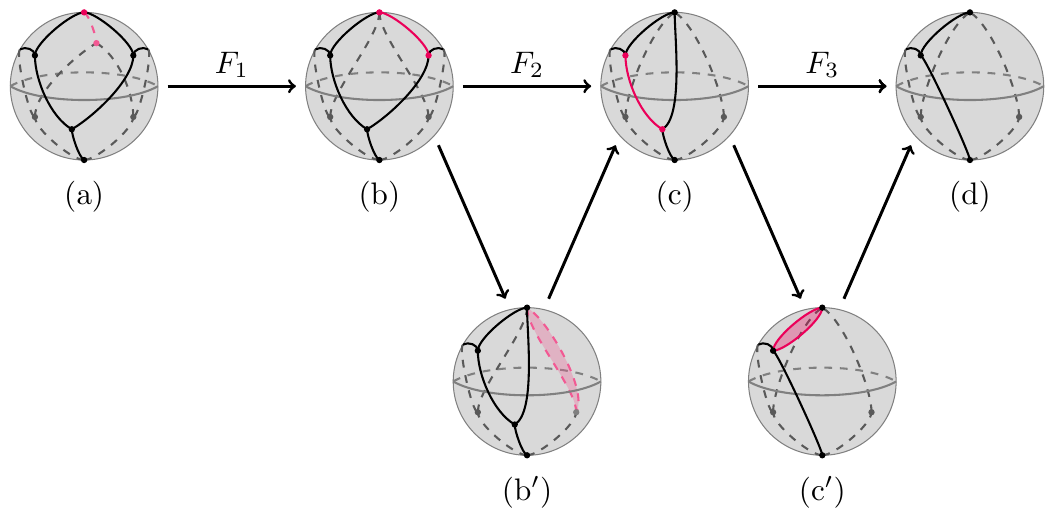}
\caption{The zipping sequence induced by the strong maps depicted in \Cref{map factorization fig}.}
\label{zip constr fig}
\end{figure}

\newcommand{\lesszip}{{<_{\text{zip}}}}

	\emph{Proof of \Cref{zip construction}:}
		Let~$f:(L,G)\rightarrow(K,H)$ be a strong surjection.
		By Lemma~\ref{strong map factorization} we may assume~$f$
		has a single nontrivial fiber~$\{x\prec y\}$.
		Let~$F:\M(L,G)\rightarrow\M(K,H)$ be the map induced by~$f$.
		By \Cref{irreducible minor maps} the nontrivial fibers of~$F$ are
		triples~$\{(M,I),(M_x,I_x),(M_y,I_y)\}$ satisfying
		Conditions~\eqref{irr minor map fiber 1} through~\eqref{irr minor map fiber 3}.
		Let~$Z$ be the set of maximums of
		these nontrivial fibers of~$F$.
		Choose some total ordering~$\lesszip$ of~$Z$ with the property that
		if~$(M_1,I_1)$ and~$(M_2,I_2)$ are elements of~$Z$
		such that~$\rk(M_1,I_1)<\rk(M_2,I_2)$
		then~$(M_1,I_1)\lesszip(M_2,I_2)$. The poset~$\M(K,H)$ is obtained
		from~$\M(L,G)$ by identifying each of the
		nontrivial fibers;
		it will be shown these identifications may be made by
		zipping operations done in the order~$\lesszip$.
				
		Consider a nontrivial fiber~$\{(M,I),(M_x,I_x),(M_y,I_y)\}$ of
		the map~$F$. Let~$P$ be the poset
		obtained from~$\M(L,G)$ by zipping elements of~$Z$ in increasing order
		with respect to~$\lesszip$ up to but not including the step of zipping
		the element~$(M,I)$.
		The minor poset~$\M(L,G)$ is graded and thin
		by Lemma~\ref{minor posets are graded and thin},
		so Proposition~\ref{zips preserve graded and thin} implies that~$P$ is graded
		and thin as well.
		Let~$\pi:\M(L,G)\rightarrow P$ be the map
		induced by the zipping operations.
		By Remark~\ref{thin zipper remark}
		to show the triple~$\pi(M_x,I_x),\pi(M_y,I_y),\pi(M,I)$
		forms a zipper in~$P$ it suffices to
		show that~$\pi(M,I)$ only covers the elements~$\pi(M_x,I_x)$
		and~$\pi(M_y,I_y)$, and show that the join~$\pi(M_x,I_x)\join\pi(M_y,I_y)$
		in~$P$ is~$\pi(M,I)$.

		We first show that~$\pi(M,I)$ only covers~$\pi(M_x,I_x)$ and~$\pi(M_y,I_y)$.
		To this end, we claim any minor~$(M',I')$
		other than~$(M_x,I_x)$ or~$(M_y,I_y)$
		covered by~$(M,I)$ satisfies~$x,y\in I'\cup\{\widehat{0}_{M'}\}$.
		This claim is obvious when~$(M',I')$ is a deletion
		of~$(M,I)$. 
		As for contractions,
		since~$x$ is only covered by~$y$, for any~$z\in L$ either~$z\join x=z\join y$
		or~$z\le x$. Thus, contracting~$(M,I)$ by an element~$\ell\ne x$
		either fixes both elements~$x$ and~$y$ or removes more than one generator.
		This proves the claim.
		By construction, any minor~$(M',I')$ such that~$x,y\in I'$
		was the maximum
		of a zipper in the construction of~$P$
		and thus~$\rk_P(\pi(M',I'))<\rk_{\M(L,G)}(M',I')$.
		Hence the element~$\pi(M',I')$ is not covered by~$\pi(M,I)$ and
		we conclude~$\pi(M,I)$ only covers the elements~$\pi(M_x,I_x)$
		and~$\pi(M_y,I_y)$ in the poset~$P$.

		It remains to show~$\pi(M,I)=\pi(M_x,I_x)\join\pi(M_y,I_y)$ in~$P$.
		To this end, we first observe~$(M,I)=(M_x,I_x)\join(M_y,I_y)$
		in~$\M(L,G)$; keeping in mind that either~$\zerohat_{M_x}=\zerohat_{M_y}$
		or~$\zerohat_{M_x}$ is only covered by~$\zerohat_{M_y}$
		it is straightforward to check using Lemma~\ref{minor poset joins}
		that
			\[(M,I)=(M_x,I_x)\join(M_y,I_y)\]
		in~$\M(L,G)$.
		Now consider an element~$p\in P$ that is an upper bound
		for the elements~$\pi(M_x,I_x)$ and~$\pi(M_y,I_y)$.
		By construction, the fibers~$\pi^{-1}\circ\pi(M_x,I_x)$
		and~$\pi^{-1}\circ\pi(M_y,I_y)$
		are trivial. The fact that~$p>\pi(M_x,I_x)$ implies there
		is some minor~$(N_x,J_x)$ of~$(L,G)$
		satisfying,
			\begin{align*}
				(N_x,J_x)&>(M_x,I_x)\\
				\pi(N_x,J_x)&=p.
			\end{align*}
		Similarly,
		there is a minor~$(N_y,J_y)$ of~$(L,G)$
		satisfying,
			\begin{align*}
				(N_y,J_y)&>(M_y,I_y)\\
				\pi(N_y,J_y)&=p.
			\end{align*}
		If~$(N_y,J_y)=(N_x,J_x)$ then~$(N_x,J_x)\ge(M,I)$, hence~$p\ge\pi(M,I)$
		as desired.
		We show this equality must hold.

		Since~$p>\pi(M_x,I_x)$ we have
			\[\rk_P(p)\ge\rk_P(\pi(M_x,I_x))+1
			=\rk(M,I).
			\]
		By construction, no minors of~$(L,G)$ of rank greater than~$\rk(M,I)$
		are the maximum of a zipper in the construction of~$P$,
		so any element~$q\in P$ whose fiber~$\pi^{-1}(q)$ is nontrivial
		satisfies~$\rk(q)<\rk(M,I)$.
		Thus, the fiber~$\pi^{-1}(p)$ is trivial which completes the proof.
	\hfill\qedsymbol

\begin{corollary}
\label{minor posets are Eulerian and PL-spheres}
	Let $(L,G)$ be a {\genlatt}.
	\begin{itemize}
		\item[(a)]{The minor poset~$\M(L,G)$ is
			Eulerian.
		}
		\item[(b)]{If~$G\ne\emptyset$ 
			the order complex
			~$\Delta(\M(L,G)\wout\{\emptyset,(L,G)\})$
			of the proper part
			of~$\M(L,G)$ is a PL-sphere. In particular,
			the minor poset $\M(L,G)$ is isomorphic to the
			face poset of a CW sphere.
		}
	\end{itemize}
\end{corollary}

	\begin{proof}
		Proposition~\ref{M(B_n) is cube} implies the minor poset~$\M(B_G,\irr(B_G))$
		is isomorphic to the face lattice~$Q_{\abs{G}}$ of the~$\abs{G}$-dimensional
		cube. In particular,
		the order complex
			\[\Delta(\M(B_G,\irr(B_G))\wout\{\emptyset,(B_G,\irr(B_G))\})\]
		is a PL-sphere.
		
		There is a strong surjection from~$(B_G,\irr(B_G))$ onto~$(L,G)$, namely the 
		canonical strong map defined in \Cref{canonical strong map def}.
		Theorem~\ref{zip construction} implies the minor poset~$\M(L,G)$
		can be obtained from the poset~$Q_{\abs{G}}$ via zipping operations.
		Theorem~\ref{zips preserve PL-sphere} implies the
		order complex~$\Delta(\M(L,G)\wout\{\emptyset,\M(L,G)\}$ is a PL-sphere,
		which establishes (b).

		Part (a) follows from (b) except for the trivial case $G=\emptyset$
		when $\M(L,G)$ is a 2~element chain.
	\end{proof}

\section{\cv\dv-index inequalities}
\label{inequality section}

In this section we deduce inequalities between \cv\dv-indices
of minor posets. In particular, we establish that the \cv\dv-indices of minor
posets of a fixed rank~$n$ are bounded above by the face lattice of
the $n$-cube. In contrast, in \cite[Corollary 7.5]{billera-ehrenborg-readdy-97}
it was shown that the face lattice of the $n$-cube achieves the minimum
\cv\dv-index across all lattices of regions of oriented matroids of
a fixed rank. We also establish in the special case of minimally
generated join extremal lattices that the minimum \cv\dv-index
is achieved by the Boolean algebra. This has some overlap with
\cite[Corollary 1.3]{gorlatts}
which states the minimum \cv\dv-index across all Gorenstein* lattices
of a fixed rank is achieved by the Boolean algebra.
Note, not every join extremal lattice has a minor poset that is itself
a lattice. For example, the minor poset of the {\genlatt}
\Cref{gl_tam_3}, which is depicted as a CW sphere
in \Cref{zip constr fig} (d), is not a lattice.

\begin{corollary}
\label{minor poset cd-index inequalities}
	Let~$(L,G)$ be a {\genlatt} with~$n$ generators. The following inequalities
	hold coefficientwise among the {\bf cd}-indices:
	\[0\le\Psi(\M(L,G))\le\Psi(Q_n)=\Psi(\M(B_n,\irr(B_n))),\]
	where $Q_n$ denotes the face lattice of the $n$-dimensional cube.
\end{corollary}

	\begin{proof}
		The left-hand inequality is implied
		by \Cref{minor posets are Eulerian and PL-spheres}
		and Theorem~\ref{cd-index nonnegativity}. Theorem~\ref{cd-index nonnegativity}
		also applies to
		the face lattice of the~$n$-dimensional cube, to the intermediate
		posets in the zipping construction of~$\M(L,G)$ and to all intervals
		of these posets; as order complexes of the proper parts of these posets
		are PL-spheres. Theorem~\ref{zip cd-index inequalities} implies
		the {\bf cd}-index of the minor poset~$\M(L,G)$ may be obtained
		from the {\bf cd}-index of the face lattice of the~$n$-dimensional
		cube by subtracting terms that all have nonnegative coefficients.
	\end{proof}

\begin{corollary}
\label{strong map cd-index inequalities}
	Let~$(L,G)$ and~$(K,H)$ be {\genlatts} such that there is a strong
	surjection from~$(L,G)$ onto~$(K,H)$. The following inequality
	of {\bf cd}-indices is satisfied coefficientwise:
	\begin{align}\label{minor posets cd-index inequality}
	\Psi(\M(K,H))\cdot {\bf c}^{\lvert{G}\rvert-\lvert{H}\rvert}
	\le\Psi(\M(L,G)).\end{align}
\end{corollary}

	\begin{proof}
		First consider the case where~$\lvert{G}\rvert=\lvert{H}\rvert$.
		By Theorem~\ref{zip construction} we have a sequence of zipping operations
		that takes the minor poset~$\M(L,G)$ to~$\M(K,H)$.
		The minor poset~$\M(L,G)$ has proper part a PL-sphere, hence
		so does every intermediate poset resulting from the
		sequence of zipping operations.
		By Theorem~\ref{cd-index nonnegativity} every intermediate poset,
		and every interval of every intermediate poset,
		has a~{\bf cd}-index with nonnegative coefficients.
		By assumption~$\rk(\M(L,G))=\rk(\M(K,H))$ so no zipping operation
		involves the maximal element.
		Thus, Theorem~\ref{zip cd-index inequalities} (a) implies each
		zipping operation corresponds, on the level of~{\bf cd}-indices,
		to subtracting off some~{\bf cd}-polynomial
		with nonnegative coefficients.
		Therefore, we have~$\Psi(\M(K,H))\le\Psi(\M(L,G))$
		coefficientwise when~$\lvert{G}\rvert=\lvert{H}\rvert$.

		Now consider the case~$\abs{G}=\abs{H}+1$ and let~$f:(L,G)\rightarrow(K,H)$
		be a strong surjection. We claim that the map~$f$ may be factored as in \Cref{strong map factorization}
		as~$f=f_r\circ\dots\circ f_1$ so that only the map~$f_r$ reduces the number
		of generators. The assumption~$\abs{G}=\abs{H}+1$ implies there
		is a generator~$g\in G$ such that~$f(g)=f(x)$ for some~$x\in G\cup\{\zerohat_L\}$
		and furthermore, the edge~$(x,x\join g)$ is minimal in~$\mcE(f)$.
		Removing the edge~$(x,x\join g)$ from the map~$f$ produces a map that does
		not change the number of generators which proves the claim.

		Set~$(M,I)=f_{r-1}\circ\dots\circ f_1(L,G)$, so that~$f_r(M,I)=(K,H)$. In the
		zipping sequence to construct~$\M(K,H)$ from~$\M(M,I)$ only the final zipping
		operation involves the maximum.
		Letting~$P$ be the poset obtained from~$\M(M,I)$
		by performing the zipping operations up to, but not including this final step,
		we have
			\[\Psi(P)=\Psi(\M(K,H))\cdot\cv.\]
		Furthermore, the same reasoning as used above for the case~$\abs{G}=\abs{H}$
		establishes $\Psi(P)\le\Psi(\M(M,I))$, hence,~$\Psi(\M(K,H))\cdot\cv\le\Psi(\M(M,I))$.
		The previous case implies
		\[\Psi(\M(K,H))\cdot\cv\le\Psi(\M(M,I))\le\Psi(\M(L,G)).\]

		The general case now follows from induction.

	\end{proof}

\begin{remark}
	It would be too much to expect the converse of \Cref{minor poset cd-index inequalities} to hold,
	that is, to expect {\bf cd}-index inequalities for minor posets to
	imply the existence of
	strong surjections between
	the associated {\genlatts}. Indeed, the converse is false.
	As a counterexample, let~$(L,G)$ be the {\genlatt} depicted
	in \Cref{gl_tam_3} and let~$(K,H)$ be the
	{\genlatt} depicted in \Cref{gl_U_2_3}.
	There is no strong surjection
	from~$(L,G)$ onto~$(K,H)$, but the inequality is satisfied since
	\begin{align*}\Psi(\M(L,G))&={\bf c}^3+2{\bf cd}+3{\bf dc},\\
	\Psi(\M(K,H))&={\bf c}^3+{\bf cd}+3{\bf dc}.\end{align*}
\end{remark}

A direct appeal to inequality (\ref{minor posets cd-index inequality}) cannot establish which {\genlatts} achieve
minimal \cv\dv-indices. For example, the lattices depicted in \Cref{gl_C_3} and \Cref{gl_B_2}
(the chain and the rank~2 Boolean algebra)
admit no strong surjection onto another {\genlatt} with~3 generators, but there is a coefficientwise
inequality since the \cv\dv-indices are $\cv^3+2\cv\dv+2\dv\cv$ and $\cv^3+\cv\dv+2\dv\cv$.

We make a conjecture as to what {\genlatts} achieve the
minimum \cv\dv-index across minor posets of a fixed rank.
Let~$A_n$ denote the lattice obtained from an antichain
of size~$n$ by adjoining a minimum and maximum and let $\mcA_n=(A_n,A_n\wout\{\zerohat\})$. The diagram of the
{\genlatt}~$\mcA_2$ is depicted in \Cref{gl_B_2}.
Among {\genlatts} with~3 generators the {\genlatt} $\mcA_2$
has minor poset
with minimum \cv\dv-index.
Similarly, the minor poset of $\mcA_3$ achieves the minimum \cv\dv-index
among minor posets of {\genlatts} with~4 generators.
This case was checked with a computer using
\Cref{join relations and lift order ideals} to enumerate all~134 {\genlatts}.
We conjecture this behavior continues.

\begin{conjecture}
\label{cd min}
	The set of~$\cv\dv$-indices of minor posets of {\genlatts} with~$n$ generators
	is coefficientwise minimized by the~$\cv\dv$-index~$\Psi(\M(\mcA_{n-1})$.
\end{conjecture}

Recall the flag~$f$-vector of a rank $n+1$ poset with maximum and minimum is the vector $f=f(P)$
with entries $f_S$ equal to the number of chains $C$ in $P$ containing the maximum and minimum and
such that $\{\rk(c):c\in C\wout\{\zerohat,\onehat\}\}=S$. Coefficientwise inequalities between \cv\dv-indices
of two posets implies entrywise inequalities between the flag $f$-vectors, but the converse does not hold in general.
We show the {\genlatts} $\mcA_n$ have minor posets with flag $f$-vectors achieving the minimum among all {\genlatts}
with $n$ generators. This supports \Cref{cd min}.

\begin{proposition}
\label{f-vector min}
	The set of flag $f$-vectors of minor posets of {\genlatts} with~$n$ generators
	is entrywise minimized by the flag $f$-vector of $\M(\mcA_n)$.
\end{proposition}

	\begin{proof}
		We give a rank set preserving injection from the set of chains in $\M(\mcA_n)$ into the set of chains in $\M(L,G)$
		for any given {\genlatt} $(L,G)$ with $n$ generators.
		To define our injection $\phi$
		fix any linear extension $g_1,\dots,g_n$, of $G$
		and fix a linear extension $a_1,\dots,a_n=\onehat$, of $A_n\wout\{\zerohat\}$. Consider a chain
		$C=\{\mcA_n>(K_1,H_1)>\dots>(K_r,H_r)>\emptyset\}$. There are six cases to consider:
			\begin{enumerate}
				\item{$(K_i,H_i)=\mcA_n\wout\{a_j:j\in J_i\}$ for $1\le i\le r$,
				in which case
				\[\phi(\mcA_n\wout\{a_j:j\in J_i\}) = (L,G)\wout\{g_j:j\in I\}\] for $1\le i\le r$.
				}
				\item{$(K_i,H_i) = \mcA_n\wout\{a_j:j\in J_i\}$ for $1\le i\le r-1$,
					in which case
					\[\phi(\mcA_n\wout\{a_j:j\in J_i\}) = (L,G)\wout\{g_j:j\in I\}\] for $1\le i\le r-1$
					and either
					\begin{enumerate}
						\item{
						$(K_r,H_r) = \langle\emptyset|\onehat_{A_n}\rangle$ in which case
						\[\phi(K_r,H_r) = \langle\emptyset|\onehat_L\rangle,\]
						}
						\item{
						$(K_r,H_r) = \langle\emptyset|a_k\rangle$ for some $k\in[n-1]$, in which case
						\[\phi(K_r,H_r) = \langle\emptyset|g_k\rangle,\]
						}
						\item{
						or $(K_r,H_r) = \langle\onehat_{A_n}|a_k\rangle$ for some $k\in [n-1]$, in which case
						\[\phi(K_r,H_r) = \langle g_n\join g_k|g_k\rangle.\]
						}
					\end{enumerate}
				}
				\item{
				$(K_i,H_i)=\mcA_n\wout\{a_j:j\in J_i\}$ for $1\le i\le r-2$ and $(K_{r-1},H_{r-1})=\langle\onehat_{A_n}|a_k\rangle$
				for some $k\in [n-1]$ in which case
				\[\phi(K_i,H_i) = (L,G)\wout\{g_j:j\in J_i\}\] for $1\le i\le r-2$ and
				$\phi(K_{r-1},H_{r-1})=\langle g_n\join g_k|g_k\rangle$. Furthermore, either
				\begin{enumerate}
					\item{$(K_r,H_r) = \langle\emptyset|a_k\rangle$ in which case
						\[\phi(K_r,H_r)=\langle\emptyset|g_k\rangle,\]
					}
					\item{or $(K_r,H_r) = \langle\emptyset|\onehat_{A_n}\rangle$ in which case
						\[\phi(K_r,H_r)=\langle\emptyset|g_n\join g_k\rangle.\]
					}
				\end{enumerate}
				}
			\end{enumerate}
	\end{proof}

We will prove a special case of \Cref{cd min}, but first we examine join extremal
lattices which satisfy a tighter lower bound.
Recall the \emph{length} of a lattice~$L$ is the length of the longest chain in~$L$.
The length is at most the number of join-irreducibles and a lattice whose length is equal to the number of
join-irreducibles is called \emph{join extremal}. Markowsky characterized join extremal lattices
in \cite[Theorem 11, Theorem 12]{extremal-latts} and showed Tamari lattices are join extremal.
In \cite{avann} Avann gave many equivalent characterizations of meet-distributive lattices and in particular
showed meet-distributive lattices are precisely ranked join extremal lattices.

\begin{lemma}
\label{long lattices surject onto chain}
	Let~$L$ be a join extremal lattice with $n$ join-irreducibles.
	There is a strong surjection
	from~$(L,\irr(L))$ onto the chain~$(C_n,\irr(C_n))$ of length~$n$.
\end{lemma}

	\begin{proof}
		Let~$\zerohat=\ell_0<\dots<\ell_n=\onehat$ be a chain in~$L$ of length~$n$.
		Define a map
			\[
			\theta:(L,\irr(L))\rightarrow(C_n,\irr(C_n))
			\]
		by
			\[
			\theta(\ell)=\min\{i\in\{0,\dots,n\}:\ell_i\ge\ell\}.
			\]
		First, we show~$\theta(\irr(L)\cup\{\zerohat_L\})=C_n$. For each~$i=1,\dots,n$ there is an irreducible~$g_i\in\irr(L)$
		such that~$\ell_{i-1}\join g_i=\ell_i$,
		hence~$\ell_i\ge g_i$ and~$\ell_{i-1}\not\ge g_i$;
		furthermore,~$\theta(\ell_i)=\theta(g_i)$.
		Thus, since~$\theta$ maps the chain~$\{\ell_0,\dots,\ell_n\}$ onto~$C_n$
		it also maps~$\irr(L)$ onto~$C_n$.

		To show~$\theta$ is join-preserving take~$\ell,k\in L$.
		The image of the join~$\ell\join k$ is
			\begin{align*}
			\theta(\ell\join k)&=\min\{i\in\{0,\dots,n\}:\ell_i\ge\ell\join k\}\\
			&=\min\{i\in\{0,\dots,n\}:\ell_i\ge\ell,\ \ell_i\ge k\}\\
			&=\max\{\theta(\ell),\theta(k)\}\\
			&=\theta(\ell)\join\theta(k).
			\end{align*}
	\end{proof}

\begin{corollary}
\label{boolean algebra minimizes no parallel cd-indices}
	The class of~{\bf cd}-indices of minor posets of {\genlatts}~$(L,\irr(L))$
	for which $L$ is join extremal with length $n$
	is coefficientwise minimized
	by the~{\bf cd}-index of the rank~$n+1$ Boolean algebra.
\end{corollary}

	\begin{proof}
		Lemma~\ref{long lattices surject onto chain} and
		\Cref{minor poset cd-index inequalities} imply
		\[\Psi(B_{n+1})=\Psi(\M(C_n,\irr(C_n)))\le\Psi(\M(L,G)).\]
	\end{proof}

\begin{proposition}
\label{cd min partial}
	Let $(L,G)$ be a {\genlatt}. If any of the following conditions hold
	then $\Psi(\M(\mcA_{\abs{G}}))\le\Psi(\M(L,G))$.
	\begin{enumerate}
		\item[(i)]{At most one generator is not an atom of $L$.}
		\item[(ii)]{$G=\irr(L)$ and $L$ is join extremal.}
		\item[(iii)]{There exists an atom $a$ of $L$ such that
			$a\not< g$ for all $g\in G$.
			}
		\item[(iv)]{There is a coatom $c$ of $L$ that is a generator
			and the interval $[\langle \onehat|c\rangle,(L,G)]$
			in the minor poset has a \cv\dv-index bounded below by
			the Boolean algebra of rank $\abs{G}-1$.
			}
	\end{enumerate}
\end{proposition}

	\begin{proof}
		In Case (i) the lattice admits a strong surjection onto $\mcA_n$
		by sending all elements except the atoms to $\onehat_{A_n}$.
		If all generators are atoms then we also choose one atom arbitrarily
		to map to $\onehat$. Thus, Case (i) follows from
		\Cref{minor poset cd-index inequalities}.

		Now we reduce Cases (ii) and (iii) to that of Case (iv).
		In Case (ii) the {\genlatt} $(L,G)$ admits a strong surjection
		onto the chain $(C_{\abs{G}},\irr(C_{\abs{G}}))$. It thus
		suffices to prove the result for chains which falls under
		Case (iv) since $\M(C_{\abs{G}},\irr(C_{\abs{G}}))$ is Boolean.

		To reduce Case (iii) to Case (iv) consider the relation $\alpha$ on $L$ defined by $\alpha(\ell)=\alpha(k)$ if
		either $\ell=k$ or $\ell>a$ and $k>a$. Since for any $x>a$ we have
		$x\join\ell >a$ for all $\ell\in L$ the relation $\alpha$ is join-preserving.
		By \Cref{minor poset cd-index inequalities} $\Psi(\M(\alpha(L,G)))\le\Psi(\M(L,G))$.
		Thus, by replacing $(L,G)$ with $\alpha(L,G)$ we may assume $a$ is both an atom
		and a coatom of $L$. The interval $[\langle\onehat_L|a\rangle,(L,G)]$
		of $\M(L,G)$ is Boolean as witnessed by the map that sends
		$X\subsetneq G\wout\{a\}$ to the deletion by $X$ and
		sends $G\wout\{a\}$ to $\langle\onehat_L|a\rangle$.

		To prove Case (iv) we consider the zipping sequence induced by the map that identifies an atom of
		$\mcA_n$ with $\onehat_{A_n}$. We will compare this zipping sequence to the sequence corresponding
		to identifying the (co)atom $a$ with $\onehat_L$. By showing that the latter
		sequence reduces the \cv\dv-index more than the former the result will follow by induction.
		The base case $n=2$ is obvious as there are only the {\genlatts} $\mcA_1=(C_2,\irr(C_2))$ and
		$(B_2,\irr(B_2))$.

		Let the atoms of $\mcA_n$ be $a_1,\dots,a_{n-1}$.
		The zippers in the zipping sequence induced by the map $\mcA_n\rightarrow\mcA_{n-1}$
		are all triples $\{(M_{a_{n-1}},I_{a_{n-1}}),(M_\onehat,I_\onehat),(M,I)\}$ such that
		$\{a_{n-1},\onehat\}\subseteq I\cup\{\zerohat_M\}$, $(M_{a_{n-1}},I_{a_{n-1}})=(M,I)\wout\{\onehat\}$
		and $(M_\onehat,I_\onehat)$ is obtained from $(M,I)$ by removing $a_{n-1}$. Such triples come from deletions
		$(M,I)$ of $\mcA_n$ containing both $\onehat$ and $a_n$ as generators and from the
		contraction $\mcA_n/\{a_{n-1}\}$. On the other hand, the sequence collapsing $(L,G)$ has zippers for deletions
		containing both $a$ and $\onehat_L$ as zippers. There is also a zipper for the
		contraction $(L,G)/\{a\}$. Possibly there are more zippers, but we will show the contribution
		from those listed exceeds the total contribution from the zippers in the sequence collapsing
		$\mcA_n$ to $\mcA_{n-1}$.

		Let $\beta$ be the relation on $(L,G)$ identifying $a$ with $\onehat$.
		Consider a deletion $(L,G)\wout X$ where $a,\onehat_L\not\in X\subseteq G$. At the point
		of zipping the triple
		\[(L,G)\wout X,\ \ \  (L,G)\wout (X\cup\{a\}),\ \ \ (L,G)\wout(X\cup\{\onehat_L\})\]
		the lower interval \[[\emptyset,(L,G)\wout (X\cup\{a\})]\]
		in the intermediate poset created by the zipping operations
		is isomorphic to the interval $[\emptyset,(L,G)\wout(X\cup\{\onehat\})]\subseteq
		\M(\beta(L,G))$. By induction, $\Psi(\M(\beta(L,G)))\ge\Psi(\M(\mcA_{n-1}))$.
		The upper interval $[(L,G)\wout X,(L,G)]$ in the zipped poset is isomorphic to the interval $[(L,G)\wout X,(L,G)]\subseteq\M(L,G)$
		of the same name. This interval is Boolean of rank $\abs{X}$. The same is true of $\mcA_n$,
		that is, the interval $[\mcA_n\wout X',\mcA_n]$ in the intermediary is Boolean of
		rank $\abs{X'}$ for $X'\subseteq\{a_1,\dots,a_{n-1}\}$. This establishes the contribution
		from all zippers with top element a deletion is bounded below by the contribution from the
		same zippers in the sequence collapsing $\mcA_n$.

		Now consider the zipper with top element $\langle \onehat|a\rangle$. The lower interval
		$[\emptyset,\langle\emptyset|\onehat\rangle]$ in the zipped poset
		is isomorphic to the interval $[\emptyset,\langle\emptyset|\onehat\rangle]\subseteq\M(\beta(L,G))$ which is Boolean of rank~2. The interval
		$[\emptyset,\langle\emptyset|\onehat_{A_{n-1}}\rangle]_{\M(\mcA_{n-1})}$ is also Boolean of rank~2.
		The upper interval $[\langle\onehat|a_n\rangle,\mcA_n]$ in $\M(\mcA_n)$ is Boolean
		of rank $n-1$ ($\mcA_n$ itself falls under Case (iii)). Since the
		interval $[\langle\onehat_L|a\rangle,(L,G)]$ in $\M(L,G)$ is
		also Boolean the zipping sequence for $\M(\mcA_n)$ reduces the
		\cv\dv-index more than that of $\M(L,G)$.
	\end{proof}

There are {\genlatts} that do not satisfy Case (iv) of \Cref{cd min partial}.
In order to establish a lower bound on the \cv\dv-index for all minor posets
one only need consider {\genlatts} $(L,G)$ for which $G=L\wout\{\zerohat\}$;
any {\genlatt} admits a strong surjection onto one of these minimal {\genlatts}
without reducing the number of generators. The known examples that do not
satisfy Case (iv) are not minimal in this sense. It is thus an interesting
question as to whether there exists a {\genlatt} $(L,G)$ for which
$G=L\wout\{\zerohat\}$ and the interval $[\langle\onehat|c\rangle,(L,G)]$
of $\M(L,G)$ is not bounded below by the \cv\dv-index of the Boolean algebra.
A negative answer to this question would resolve \Cref{cd min} in the affirmative
while a positive answer is not enough to refute it.

\section{Acknowledgments}
	The author thanks Richard Ehrenborg and Margaret Readdy for comments. The majority of this research was supported by the University of Kentucky. The author also thanks the University of Minnesota for support.
	The python package Posets \cite{posets-package} was used to generate \Cref{3-cycle minor poset fig,tam 3 minor poset fig} and to check \Cref{cd min}.

\nocite{birkhoff}
\nocite{edmondsReprint}
\bibliography{bib}{}

\begin{thebibliography}{10}

\bibitem{avann}
S.~P. Avann.
\newblock Application of the join-irreducible excess function to semi-modular
  lattices.
\newblock {\em Math. Ann.}, 142:345--354, 1960/61.

\bibitem{bayer}
Margaret~M. Bayer.
\newblock The {$cd$}-index: a survey.
\newblock In {\em Polytopes and discrete geometry}, volume 764 of {\em Contemp.
  Math.}, pages 1--19. Amer. Math. Soc., [Providence], RI, [2021] \copyright
  2021.

\bibitem{newindex}
Margaret~M. Bayer and Andrew Klapper.
\newblock A new index for polytopes.
\newblock {\em Discrete Comput. Geom.}, 6(1):33--47, 1991.

\bibitem{billera-ehrenborg-readdy-97}
Louis~J. Billera, Richard Ehrenborg, and Margaret Readdy.
\newblock The {$c$}-{$2d$}-index of oriented matroids.
\newblock {\em J. Combin. Theory Ser. A}, 80(1):79--105, 1997.

\bibitem{birkhoff}
Garrett Birkhoff.
\newblock {\em Lattice theory}.
\newblock American Mathematical Society Colloquium Publications, Vol. XXV.
  American Mathematical Society, Providence, R.I., third edition, 1967.

\bibitem{bjorner}
A.~Bj\"orner.
\newblock Posets, regular {CW} complexes and {B}ruhat order.
\newblock {\em European J. Combin.}, 5(1):7--16, 1984.

\bibitem{crapo-strongmaps}
Henry~H. Crapo.
\newblock Structure theory for geometric lattices.
\newblock {\em Rend. Sem. Mat. Univ. Padova}, 38:14--22, 1967.

\bibitem{edmonds}
Jack Edmonds.
\newblock Submodular functions, matroids, and certain polyhedra.
\newblock In {\em Combinatorial {S}tructures and their {A}pplications ({P}roc.
  {C}algary {I}nternat. {C}onf., {C}algary, {A}lta., 1969)}, pages 69--87.
  Gordon and Breach, New York, 1970.
\newblock reprinted in \cite{edmondsReprint}.

\bibitem{edmondsReprint}
Jack Edmonds.
\newblock Submodular functions, matroids, and certain polyhedra.
\newblock In {\em Combinatorial optimization---{E}ureka, you shrink!}, volume
  2570 of {\em Lecture Notes in Comput. Sci.}, pages 11--26. Springer, Berlin,
  2003.

\bibitem{gorlatts}
Richard Ehrenborg and Kalle Karu.
\newblock Decomposition theorem for the {${\bf cd}$}-index of {G}orenstein
  posets.
\newblock {\em J. Algebraic Combin.}, 26(2):225--251, 2007.

\bibitem{thesis}
William Gustafson.
\newblock {\em Lattice minors and Eulerian posets}.
\newblock PhD thesis, University of Kentucky, 2023.
\newblock \url{https://uknowledge.uky.edu/math_etds/96}.

\bibitem{latticeMinorsAndPolymatroids}
William Gustafson.
\newblock Polymatroids, closure operators and lattices.
\newblock {\em Order}, 40(2):311--325, 2023.

\bibitem{posets-package}
William Gustafson.
\newblock Posets - a python package for finite partially ordered sets.
\newblock \url{https://www.github.com/WilliamGustafson/posets}, 2024.
\newblock Version 0.24.11.30.

\bibitem{herzog-hibi-02}
J\"urgen Herzog and Takayuki Hibi.
\newblock Discrete polymatroids.
\newblock {\em J. Algebraic Combin.}, 16(3):239--268, 2002.

\bibitem{higgs}
D.~A. Higgs.
\newblock Strong maps of geometries.
\newblock {\em J. Combinatorial Theory}, 5:185--191, 1968.

\bibitem{Hudson}
J.~F.~P. Hudson.
\newblock {\em Piecewise linear topology}.
\newblock W. A. Benjamin, Inc., New York-Amsterdam, 1969.
\newblock University of Chicago Lecture Notes prepared with the assistance of
  J. L. Shaneson and J. Lees.

\bibitem{karu}
Kalle Karu.
\newblock The {$cd$}-index of fans and posets.
\newblock {\em Compos. Math.}, 142(3):701--718, 2006.

\bibitem{electroids}
Thomas Lam.
\newblock Electroid varieties and a compactification of the space of electrical
  networks.
\newblock {\em Adv. Math.}, 338:549--600, 2018.

\bibitem{extremal-latts}
George Markowsky.
\newblock Primes, irreducibles and extremal lattices.
\newblock {\em Order}, 9(3):265--290, 1992.

\bibitem{reading}
Nathan Reading.
\newblock The cd-index of {B}ruhat intervals.
\newblock {\em Electron. J. Combin.}, 11(1):Research Paper 74, 25, 2004.

\bibitem{s-shellable}
Richard~P. Stanley.
\newblock Flag {$f$}-vectors and the {$cd$}-index.
\newblock {\em Math. Z.}, 216(3):483--499, 1994.

\bibitem{stanley}
Richard~P. Stanley.
\newblock {\em Enumerative combinatorics. {V}olume 1}, volume~49 of {\em
  Cambridge Studies in Advanced Mathematics}.
\newblock Cambridge University Press, Cambridge, second edition, 2012.

\end{thebibliography}
\bibliographystyle{plain}
\appendix
\hypertarget{zipping-appendix}{}
\section*{Appendix}
\setcounter{section}{1}
\renewcommand{\thesection}{\Alph{section}}
This appendix contains background on the zipping operation introduced by Reading
used in the proof of \Cref{zip construction}. We also review some basic PL-topology as it
relates posets and give a brief presentation of the \cv\dv-index.

A simplicial complex~$\Delta$ is a \emph{PL-sphere} when there is a piecewise
linear homeomorphism from~$\Delta$ to the boundary of a simplex.
A basic fact of PL-topology is that order complex of the proper part of the face lattice of any
polytope is a PL-sphere.

The \emph{link} of a face~$X$ in a simplicial complex~$\Delta$
is the subcomplex
\[\link_\Delta(X)=\{Y\in\Delta:Y\cap X=\emptyset\text{ and }Y\cup X\in\Delta\}.\]

\begin{lemma}[{Hudson \cite[Corollary 1.16]{Hudson}}]
\label{PL-sphere links are PL-spheres}
	Given a simplicial complex~$\Delta$ that is a PL-sphere,
	for any face~$X$ of~$\Delta$ the link~$\link_\Delta(X)$ is a PL-sphere.
\end{lemma}

Let~$P$ be a poset with~$\widehat{0}$ and~$\widehat{1}$
where the proper part is a PL-sphere and let~$[x,y]$ be an interval
of~$P$. The order complex~$\Delta((x,y))$ is the link~$\link_{\Delta(P)}(C)$
where~$C$ is the union of a maximal chain in~$[\widehat{0},x]$ and a maximal
chain in~$[y,\widehat{1}]$. Thus, by Lemma~\ref{PL-sphere links are PL-spheres}
the order complex~$\Delta((x,y))$ is a PL-sphere.
In summation, if~$P$
is a poset such that the proper part is a PL-sphere then the proper part of
any closed interval of~$P$ is a PL-sphere as well. In particular, any poset
whose proper part is a PL-sphere is Eulerian.

The following definition and theorem due to Bj\"orner
characterize face posets of regular CW complexes.

\begin{definition}[{Bj\"orner \cite[Definition 2.1]{bjorner}}]
\label{CW poset}
	A poset~$P$ is a CW poset if the following three
	conditions hold:
		\begin{enumerate}
			\item{$P$ has a minimum~$\widehat{0}$
			and a maximum~$\widehat{1}$.}
			\item{$\lvert{P}\rvert\ge3$.}
			\item{$\Delta((\widehat{0},p))$ is homeomorphic to a sphere
			for all elements~$p$ in the open interval~$(\widehat{0},\widehat{1})$.}
		\end{enumerate}
\end{definition}

The definition we give differs slightly from the one given in
\cite{bjorner} since in the present context the maximum of face posets
does not correspond to a cell.

\begin{theorem}[{Bj\"orner \cite[Proposition 3.1]{bjorner}}]
\label{CW posets are face posets}
	A poset is a CW poset if and only if it is isomorphic
	to the face poset of a regular CW complex.
\end{theorem}

A graded poset~$P$ with~$\widehat{0}$ and~$\widehat{1}$ is said to
be \emph{Eulerian} if for all~$x<y$ in~$P$,
	\[\sum_{x\le z\le y}(-1)^{\rk(z)}=0.\]
Equivalently, the M\"obius function of~$P$
satisfies~$\mu(x,y)=(-1)^{\rk(y)-\rk(x)}$ for all $x\le y$. Face lattices of polytopes are the motivating
examples of Eulerian posets. More generally, face posets of regular CW spheres
are Eulerian posets.

The zipping operation was introduced by Reading in \cite[Section 4]{reading} as a tool
to study Bruhat intervals.

\begin{definition}
\label{zip def}
	Let~$P$ be a poset and let~$x,y,z\in P$. The triple~$x,y,z$ is said to form a
	\emph{zipper} if the following three conditions hold:
	\begin{enumerate}
		\item[(i)]{$z$ covers only the two elements~$x$ and~$y$.\label{zip condition 1}}
		\item[(ii)]{$z=x\join y$.}\label{zip condition 2}
		\item[(iii)]{$\{p\in P:p< x\}=\{p\in P:p<y\}$.}\label{zip condition 3}
	\end{enumerate}
\end{definition}

Let~$P$ be a poset with a zipper~$x,y\prec z$. 
The poset~$\zip(P,z)$ has underlying set obtained from~$P$ by replacing
the elements~$x,y$ and~$z$ with a new element~$w$.
The order relation of~$\zip(P,z)$ is defined by the following three conditions
that hold for all~$p$ and~$q$ in~$P\setminus\{x,y,z\}$.
	\begin{enumerate}
	\item[\eqlabel{zipped poset 1}]
		{$p\le w$ in~$\zip(P,z)$ if and only if~$p\le z$ in~$P$.}
	\item[\eqlabel{zipped poset 2}]
		{$w\le p$ in~$\zip(P,z)$ if and only if~$x\le p$ or~$y\le p$ in~$P$.}
	\item[\eqlabel{zipped poset 3}]
		{$p\le q$ in~$\zip(P,z)$ if and only if~$p\le q$ in~$P$.}
	\end{enumerate}
When~$P$ and~$\zip(P,z)$ are graded, the rank functions
satisfy~$\rk_{\zip(P,z)}(w)=\rk_P(z)-1$.

\begin{figure}
	\def\r{0.25}
	\centering
	\hfill
	\begin{subfigure}[t]{0.45\textwidth}
		\centering
		\includegraphics[width=\r\textwidth]{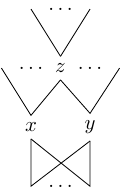}
		\caption{The Hasse diagram of~$P$.}
	\end{subfigure}
	\hfill
	\begin{subfigure}[t]{0.45\textwidth}
		\centering
		\includegraphics[width=\r\textwidth]{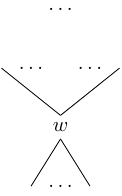}
		\caption{The Hasse diagram of the zipped poset~$\zip(P,z)$.}
	\end{subfigure}
	\hfill

	\vspace{\baselineskip}
	
	\hfill
	\begin{subfigure}[t]{0.45\textwidth}
		\centering
		\includegraphics[width=\r\textwidth]{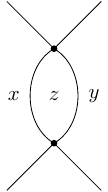}
		\caption{A cell complex with a 2-dimensional zipper.}
	\end{subfigure}
	\hfill
	\begin{subfigure}[t]{0.45\textwidth}
		\centering
		\includegraphics[width=\r\textwidth]{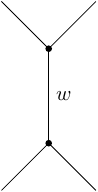}
		\caption{The zipped cell complex.}
	\end{subfigure}
	\hfill
\caption{A schematic picture of a zipper in a poset and in a cell complex.}
\end{figure}

Recall, a graded poset is said to be \emph{thin} if all length two intervals
are isomorphic to the Boolean algebra~$B_2$.

\begin{remark}
\label{thin zipper remark}
	If~$P$ is a thin poset then~$x,y\prec z$ form a zipper
	whenever conditions~(i) and~(ii) in~Definition~\ref{zip def} are satisfied.
	Condition~(iii) in the same definition follows from thinness and condition~(i) (\cite[Proposition 4.8]{reading}).
\end{remark}

\begin{proposition}[{Reading \cite[Proposition 4.4]{reading}}]
\label{zips preserve graded and thin}
	If~$P$ is a graded and thin poset with a zipper~$x,y\prec z$ then the
	poset~$\zip(P,z)$ is graded and thin as well.
\end{proposition}

\begin{theorem}[{Reading \cite[Theorem 4.7]{reading}}]
\label{zips preserve PL-sphere}
	If the order complex~$\Delta(P\wout\{\zerohat,\onehat\})$ of the proper
	part of~$P$ is a PL-sphere and~$x,y\prec z$ form a zipper then
	the order complex~$\Delta(\zip(P,z)\wout\{\zerohat,\onehat\})$ is a
	PL-sphere as well.
\end{theorem}

Now we recall the definition of the \cv\dv-index. For a further discussion
of the {\bf cd}-index see the survey by Bayer \cite{bayer}.

Let~$P$ be a graded poset of rank~$n+1$ with~$\widehat{0}$ and~$\widehat{1}$.
Let~${\bf a}$ and~${\bf b}$ be noncommuting variables of degree~1.
For each chain~$C=\{\widehat{0}<x_1<\dots<x_k<\widehat{1}\}$
in~$P$ define a weight~$w(C)=w_1\cdots w_n$
by \[w_i=\begin{cases}{\bf b}&\text{if there is a rank~$i$ element in }C,\\
{\bf a}-{\bf b}&\text{otherwise.}\end{cases}\]
The~\emph{{\bf ab}-index} of~$P$ is the polynomial
\[\Psi(P)=\sum_{C}w(C),\]
where
the sum above is over all chains~$C=\{\widehat{0}<x_1<\dots<x_k<\widehat{1}\}$.
Define noncommutative variables~${\bf c}={\bf a}+{\bf b}$ and~${\bf d}={\bf ab}+{\bf ba}$
of degree~1 and~2, respectively.
If the {\bf ab}-index of~$P$ can be expressed in terms of~${\bf c}$ and~${\bf d}$,
the resulting polynomial is called the \emph{{\bf cd}-index} of~$P$
and is also denoted~$\Psi(P)$.
Not every graded poset has a {\bf cd}-index, but every Eulerian poset has a {\bf cd}-index
\cite[Theorem 4]{newindex}.

Reading gave a expression for how zipping operations affect the~{\bf cd}-index.
We will use this result to derive inequalities between~{\bf cd}-indices of
minor posets in \Cref{minor poset cd-index inequalities}
and \Cref{strong map cd-index inequalities}.

\begin{theorem}
\label{zip cd-index inequalities}
	Let~$P$ be an Eulerian poset and suppose~$x,y\prec z$ form a zipper in~$P$.
	\begin{enumerate}
		\item[(a)]{
		({Reading \cite[Theorem 4.6]{reading}})
		If~$z\ne\widehat{1}_P$ then
		\[\Psi(\zip(P,z))=\Psi(P)-\Psi([\widehat{0},x]_P)\cdot {\bf d}\cdot
		\Psi([z,\widehat{1}]_P).\]
		}
		\item[(b)]
		{({Stanley \cite[Lemma 1.1]{s-shellable}})
		If~$z=\widehat{1}_P$ then\[\Psi(\zip(P,z))=\Psi(P)\cdot{\bf c}.\]
		}
	\end{enumerate}
\end{theorem}

In order to deduce the \cv\dv-index inequalities in \Cref{minor poset cd-index inequalities}
and \Cref{strong map cd-index inequalities} from \Cref{zip cd-index inequalities},
we must establish that the \cv\dv-indices involved have nonnegative coefficients.
Nonnegativity of the \cv\dv-index was proved for Gorenstein* posets
by Karu. In particular, a poset for which the order complex of its proper part is a PL-sphere
is a Gorenstein* poset.

\begin{theorem}[{Karu \cite[Theorem 1.3]{karu}}]
\label{cd-index nonnegativity}
	If~$P$ is a Gorenstein* poset then the coefficients
	of the {\bf cd}-index~$\Psi(P)$ are nonnegative.
\end{theorem}

\end{document}